\documentclass[final,1p,times]{elsarticle} 

\usepackage{lineno}
\nolinenumbers

\journal{Engineering Analysis with Boundary Elements}

\usepackage{multirow}
\usepackage{amsmath,amssymb}
\usepackage{amsthm}
\usepackage{bm}
\usepackage{multicol}
\usepackage{tikz}

\usepackage{algorithm}
\usepackage{algorithmic}

\usepackage{graphicx,color}
\usepackage{here}
\usepackage{subcaption}
\usepackage{siunitx}
\usepackage{mathrsfs}

\usepackage{listings}
\usepackage{color}

\def\diff{\mathrm{d}}
\def\dS{\mathrm{d}S}
\def\vy{\bm{y}}
\def\vx{\bm{x}}
\def\vt{\bm{f}} 
\def\vtcoef{f}

\def\vn{\bm{n}}
\def\vk{\bm{k}}
\def\inc{\mathrm{inc}}
\def\E{\bm{E}}
\def\Einc{\bm{E}^\inc}
\def\H{\bm{H}}
\def\Hinc{\bm{H}^\inc}
\def\J{\bm{J}}
\def\M{\bm{M}}
\def\p{\mathrm{p}}
\def\imath{\mathrm{i}}
\def\e{\mathrm{e}}
\def\sdiv{\mathrm{div}_{\rm S}}
\def\Gp{G^{\rm p}}

\usepackage{pmbb-sym}

\def\nd{n_{\rm d}} 
\def\mb{\overline{m}} 
\def\qb{\overline{q}} 
\def\Sb{E} 
\def\Np{\bm{N}^\p} 
\def\Mp{\bm{M}^\p} 

\newtheorem{remark}{Remark}
\newtheorem{lemma}{Lemma}
\newtheorem{theorem}{Theorem}

\usepackage[colorlinks,bookmarksopen,bookmarksnumbered,citecolor=red,urlcolor=red,breaklinks=true]{hyperref}

\def\equationautorefname~#1\null{%
  (#1)\null
}
\def\algorithmautorefname~#1\null{%
  Algorithm~#1\null
}
\def\sectionautorefname~#1\null{%
  Section~#1\null
}
\def\subsectionautorefname~#1\null{%
  Subsection~#1\null
}
\def\subsubsectionautorefname~#1\null{%
  Subsubsection~#1\null
}
\def\tableautorefname~#1\null{%
  Table~#1\null
}
\def\figureautorefname~#1\null{%
  Figure~#1\null
}
\def\lineautorefname~#1\null{%
  Line~#1\null
}

\usepackage{todonotes}
\usepackage{enumitem}

\usepackage{mathabx} 

\usepackage{longtable}

\def\blue#1{\textcolor{blue}{#1}}

\begin{document}


\title{An isogeometric boundary element method for three-dimensional doubly-periodic layered structures in electromagnetics}

\begin{frontmatter}

\author[NU]{Toru Takahashi\corref{cor}}
\ead{toru.takahashi@mae.nagoya-u.ac.jp}
\author[NU,OLD]{Tetsuro Hirai}
\cortext[cor]{Corresponding author}
\author[NU]{Hiroshi Isakari}
\author[NU]{Toshiro Matsumoto}

\address[NU]{Department of Mechanical Systems Engineering, Nagoya University, Furo-cho, Nagoya, Aichi, 464-8603 Japan}
\address[OLD]{Former graduate student}

\begin{abstract}
This paper proposes an isogeometric boundary element method (IGBEM) to solve the electromagnetic scattering problems for three-dimensional doubly-periodic multi-layered structures. The main concerns are the constructions of (i)~an open surface (between two layers) and (ii)~a vector basis function with using the B-spline functions. Regarding (i), we considered an algorithm to generate a doubly-periodic open surface with the tensor product of the B-spline functions of any degree. Regarding (ii), we employed the vector basis function based on the B-spline functions, which was proposed by Buffa et al~\cite{buffa2010}, and adapted it to the underlying periodic problems so that it can satisfy the quasi-periodic condition on the boundary of an open surface. The proposed IGBEM worked for solving some numerical examples satisfactorily and proved the applicability to plasmonic simulations.
\end{abstract}


\end{frontmatter}


\section*{List of symbols}

\begin{longtable}[c]{lp{.55\textwidth}p{.16\textwidth}}

  Symbol & Explanation & Related item(s)\\
  
  $n_{\rm d}$ & Number of dielectric layers (domains). & \autoref{s:problem}\\
  
  $\widetilde{D}_d$ & $d$-th dielectric layer, where $d=0,\ldots,{\nd-1}$. & \autoref{s:problem}\\
  
  $L_1$, $L_2$ & Periods in the $x_1$ and $x_2$ directions. & \autoref{s:problem}\\
  
  $\varepsilon_d$, $\mu_d$ & Permittivity and magnetic permeability in $\widetilde{D}_d$.  & \autoref{s:problem}\\
  
  $\omega$ & Angular frequency. & \autoref{eq:incident_EH}\\
  
  $t$ & Time. & \autoref{eq:incident_EH}\\
  
  $\imath$ & Imaginary unit, i.e. $\sqrt{-1}$. & \autoref{eq:incident_EH}\\
  
  $k_d$ & Wavenumber in the $\widetilde{D}_d$, where $d=0,\ldots,{\nd-1}$. & \autoref{eq:incident_EH}, \autoref{eq:vkn}\\
  
  $\Einc$, $\Hinc$ & Incident time-harmonic electromagnetic fields given in $\widetilde{D}_0$. & \autoref{eq:incident_EH}\\
  
  $\vk^\inc$ & Incident wavenumber vector. & \autoref{eq:incident}\\

  $\theta$, $\phi$ & Angles of the incident wavenumber vector. & \autoref{eq:incident}\\

  $D$ & Primitive cell. & \autoref{eq:primitive_cell}\\
  
  $S^\p$ & Periodic boundary. & \autoref{eq:Sp}\\
  
  $D_d$ & Dielectric layers (domains) involved in $D$, i.e. $D_d:=\widetilde{D}_d\cap D$, where $d=0,\ldots,{\nd-1}$. & \autoref{s:problem}\\
  
  $S_d$ & Interface between $D_d$ and $D_{d+1}$, i.e. $S_d := \partial D_d \cap \partial D_{d+1}$, where $d=0,\ldots,n_{\rm d}-2$. The subscript $d$ is omitted if no confusion occurs. & Subsections~\ref{s:regularisation}, \ref{s:buffa}, \ref{s:qpv_basis}, \ref{s:integral}.\\
  
  $\E_d$, $\H_d$ & Time-harmonic electromagnetic fields in $D_d$ & \autoref{eq:maxwell1}, \autoref{eq:maxwell2}\\
  
  $\bm{n}_d$ & Unit outward normal vector of the boundary $\partial D_d$. & \autoref{eq:qpbBC1}, \autoref{eq:qpbBC2}\\
  
  $\J_d$, $\M_d$ & Surface electric and magnetic current densities in terms of $\partial D_d$. & \autoref{eq:qpbBC1}, \autoref{eq:qpbBC2}\\
  
  $\beta_1$, $\beta_2$& Phase differences. & \autoref{eq:qpb1}, \autoref{eq:qpb2}\\
  
  $\mathscr{L}^\p_d$, $\mathscr{K}^\p_d$ & Boundary integral operators in terms of $D_d$, where $d=0,\ldots,{\nd-1}$ & \autoref{eq:opL} and \autoref{eq:opK}\\
  
  $\vt_d$ & Surface current density $\J_d$ or $\M_d$ in terms of $D_d$. The subscript $d$ is omitted if no confusion occurs. & \autoref{eq:opL}, \autoref{eq:opK}, \autoref{eq:Pdivconf} etc\\
  
  $\Gp_d$ & Periodic Green's function for the wavenumber $k_d$. & \autoref{eq:Gp}\\
  
  $G_d$ & Fundamental solution for the 3D Helmholtz equation for the wavenumber $k_d$. & \autoref{eq:Gp}\\
  
  $\bm{p}^{(\bm{\nu})}$ & Translation vector in $\Gp_d$, where $\bm{\nu}\in\bbbz^2$. & \autoref{eq:Gp}\\
  
  $\bm{w}$ & Vector weight function. & \autoref{eq:mom}\\
  
  $B_i^p$ & B-spline function of degree $p$, where $i=0,\ldots,n-1$. & \autoref{eq:CoxdeBoor}\\
  
  $t_i$& Knots of B-spline functions, where $i=0,\ldots,n+p$. & \autoref{eq:CoxdeBoor}\\
  
  $\bm{p}_i$ ($\in\bbbr^2$) & Control points for a (periodic) B-spline curve, where $i=0,\ldots,n-1$. & \autoref{eq:Bcurve}\\
  
  $n_h$, $p_h$ & Parameters for a (periodic) B-spline surface for the coordinate $t_h$, where $h=1,2$. & \autoref{eq:pBsurf_x}\\
  
  $t_{h,i}$ & Knots of a (periodic) B-spline surface for the coordinate $t_h$, where $i=0,\ldots,n_h+p_h$. & \autoref{eq:pBsurf_x}\\
  
  $\bm{p}_{i,j}$ ($\in\bbbr^3$) & Control points for a (periodic) B-spline surface, where $i=0,\ldots,n_1-1$ and $j=0,\ldots,n_2-1$.  & \autoref{eq:pBsurf_x}\\
  
  $\Sb_i$ & B\'ezier element. & Second item in Remark~\autoref{theo:pBsurfprop}\\
  
  $\bm{V}_{h,i,j}$ & Buffa's vector basis function~\cite{buffa2010}. & \autoref{eq:compatibleB}\\
  
  $u_{h,i}$ & Knots of a vector basis function for the coordinate $u_h$, where $i=0,\ldots,m_h-1$. & \autoref{eq:compatibleB}\\
  
  $m_h$, $q_h$ & Parameters of a vector basis function for the coordinate $u_h$, where $h=1,2$. & \autoref{eq:compatibleB}\\
  
  $\mb_{2h-1}$, $\qb_{2h}$ & Parameters associated with $m_1$, $m_2$, $q_1$, and $q_2$. & \autoref{eq:mb_qb}\\
  
  $\Mp_{h,i,j}$, $\Np_{h,i,j}$ & Quasi-periodic vector basis functions. The subscript can be written as a single index or omitted if unnecessary. & \autoref{eq:pMdiv}, \autoref{eq:pNdiv}, \autoref{eq:approx_current}\\
  
\end{longtable}

\section{Introduction}\label{s:introduction}

The isogeometric analysis (IGA), which is a class of isoparametric formulation that employs the NURBS (including B-spline) function as shape and basis (approximation) functions in the discretisation process\footnote{This is not exactly the case to 3D electromagnetic problems under consideration. In fact, the shape and basis functions are not exactly the same. However, the term `isogeometric' is widely used nowadays when those functions are related to the NURBS function.}, has drawn academic and industrial attentions from the fields of the finite and boundary element methods (FEM and BEM) since the first paper by Hughes et al.~\cite{hughes2005}. The major advantage of the IGA over the conventional piecewise-polynomial-based (or Lagrange) discretisations is that, once the surface (boundary) of the analysis model is represented with the NURBS surface(s) (by using a CAD software or a surface modeller) in the IGA, the discretisation of the surface can be performed exactly.

The BEM suits to the IGA better than the FEM because the discretisation of the domain surrounded by NURBS surfaces is tricky in the FEM. However, because incorporating a new shape or basis function to the BEM can bring some difficulties, in particular, the evaluation of the (near-)singular integrals, the development of the isogeometric BEM (IGBEM) is laborious relatively to that of the isogeometric FEM. Overcoming such difficulties, the IGA is being gradually applied to the BEM for various types of boundary value problems. A short survey of the IGBEM is found in \cite{takahashi2019}.

Regarding the three-dimensional (3D) electromagnetics, Buffa et al.~\cite{buffa2010} proposed a vector basis function based on the B-spline functions, which can be regarded as a generalisation of the so-called rooftop basis function~\cite{glisson1980}. This enabled to construct the IGBEM (or isogeometric method of moment) for the 3D electromagnetic scattering problems~\cite{buffa2014,evans2013,simpson2018}. In particular, Simpson et al.~\cite{simpson2018} clarified the implementation of the IGBEM for both the electric and magnetic field integral equations (EFIE and MFIE), with considering an acceleration by the $\mathcal{H}$-matrix method. Following these works, D{\"o}lz et al.~\cite{dolz2018} discussed the mathematical details (such as the existence and uniqueness of solution) in the isogeometric discretisation. Further, D{\"o}lz et al.~\cite{dolz2018numerical} compared the accuracy of the IGBEM with that of the (Galerkin) BEM based on the high-order Raviart-Thomas (RT) basis function~\cite{peterson2005}. Recently, Wolf~\cite{wolf2020}, who is the last author of \cite{dolz2018,dolz2018numerical}, described the IGBEM comprehensively from both the mathematical and numerical viewpoints. It should be noted that all these investigations are for non-periodic problems. As of now, any periodic problems have not been studied in the context of the 3D electromagnetic IGBEM as far as we know.

We thus challenged to construct an IGBEM for doubly-periodic boundary value problems (BVPs) in 3D. This is not only from our academic curiosity but also for the potential applications, such as photonic~\cite{joannopoulos2008} and plasmonic crystals~\cite{maier2007}; in particular, we are interested in the analysis and design of ultra-thin photovoltaic devices~\cite{atwater2010,polman2016}. In addition, the application for our approach could include, for example, the analysis of ground penetrating radar~\cite{warren2016} and the assessment of human body exposure to electromagnetic wave~\cite{zradzinski2019}. To handle a variety of applications, we consider a multi-layered structure, where two or more dielectric materials are stacked perpendicularly and each surface (interface) between two materials (layers) are periodic horizontally.

In addition to the aforementioned superiority of the BEM over the FEM, the BEM is more suitable than the FEM as well as other volume-type solvers such as the finite difference (time-domain) method because the top and bottom layers are unbounded in a multi-layered structure, although approximations called absorbing boundary conditions such as perfectly matched layers (PMLs)~\cite{johnson2021notes} help the application of volume-type solvers. As per, there are a number of studies on the boundary element analyses for the 2D singly-periodic multi-layer problems; see Cho et al.~\cite{cho2015} and the references therein.

On the other hand, 3D doubly-periodic multi-layer problems have been rarely studied, except for Barnes~\cite{barnes2003electromagnetic}, Otani et al.~\cite{otani2008}, and Nicholas~\cite{nicholas2008}; their approaches are based on integral equations, but the discretisations are not isogeometric. Barnes~\cite{barnes2003electromagnetic} and Nicholas~\cite{nicholas2008} use the M\"uller integral equations, whereas Otani et al.~\cite{otani2008} use the same PMCHWT formulation as our study but a different basis function from ours, i.e. the standard first-order Rao-Willton-Glisson (RWG) basis function~\cite{rao1982}. The main purpose of \cite{otani2008} is a development of a periodic fast multipole method (pFMM). We emphasise that all these studies \cite{barnes2003electromagnetic,otani2008,nicholas2008} are not involved in the IGA.

In order to develop an IGBEM for such 3D doubly-periodic multi-layered structures, we need to construct (i) a doubly-periodic open surface (in the primary cell), which is rectangular in the parameter space, and (ii) a vector basis function (as well as the weight function) that satisfies the requirement to regularise the variational integral equations. These are addressed with the B-spline function in accordance with the isogeometric concept.

This paper builds on several original works. We clarify our contribution in terms of the above requisites of (i) and (ii). Regarding (i), i.e. modelling an open surface with the B-spline function, we extend the algorithm by Shimba et al.~\cite{shimba2015}, which can generate an open curve with the B-spline function of degree 2 so that the curve can represent the unit of a periodic (and infinitely long) curve on a plane. We modify the Shimba's algorithm to handle arbitrary degree and apply the modified algorithm to generating an open surface in 3D through the tensor product. We will term such an open surface a periodic B-spline surface in \autoref{s:surf_pBsurface}.

Regarding (ii), i.e. the construction of a vector basis function, we essentially exploited the vector basis function proposed by Buffa et al.~\cite{buffa2010}. Following the notations in \cite{simpson2018} basically, we modified the vector basis function (as well as the weight function) so that it can satisfy the quasi-periodic condition, which is required to regularise the variational integral equations, on a periodic B-spline surface. Taking account of the quasi-periodicity into the Buffa's vector basis function is similar to that into the RWG basis function~\cite{rao1982}, which was mentioned by Otani et al.~\cite{otani2008} and well examined by Hu et al.~\cite{hu2011}. In addition, in the case of the 2D Helmholtz equation, Shimba et al.~\cite{shimba2015} incorporated the quasi-periodic condition into the B-spline basis function in a similar way.

The remaining part of this paper is organised as follows: \autoref{s:formulation} formulates the periodic problem to be solved and presents the corresponding boundary integral equations. Moreover, the requirements for basis and testing functions are mentioned. \autoref{s:surface} shows the way to model each open surface (interface) with a B-spline surface, with considering the periodicity. \autoref{s:basis} proposes two types of vector basis functions satisfying the quasi-periodic condition. In \autoref{s:igbem}, we establish an IGBEM with mentioning the evaluation of double-surface integrals. In \autoref{s:num} as well as \autoref{s:plasmonics}, we solve some numerical problems by our IGBEM in order to validate its accuracy and applicability to plasmonic simulations.

\section{Formulation}\label{s:formulation}

\subsection{Problem statement}\label{s:problem}

Let us consider a set of $\nd$ dielectric layers (domains) $\widetilde{D}_0,\ldots,\widetilde{D}_{\nd-1}$ stacked along the $x_3$ direction, where $\widetilde{D}_0$ and $\widetilde{D}_{\nd-1}$ denote the top and bottom layers, respectively (\autoref{fig:periodicproblem}). Suppose that $\widetilde{D}_d$ is doubly-periodic with the period of $L_1$ and $L_2$ in the $x_1$ and $x_2$ directions, respectively. Here, the permittivity and magnetic permeability of $\widetilde{D}_d$ are denoted by $\varepsilon_d$ and $\mu_d$, respectively. We suppose that $\varepsilon_d$ and $\mu_d$ are real otherwise stated.

We consider the time-harmonic electromagnetic fields when the following incident electromagnetic wave of angular frequency $\omega$ is given to the top layer $\widetilde{D}_0$:
\begin{eqnarray}
  \Einc(\vx,t) = \bm{a}^\inc\mathrm{e}^{\imath\vk^\inc\cdot\vx }\mathrm{e}^{-\imath\omega t},\qquad \Hinc(\vx,t) = \bm{b}^\inc\mathrm{e}^{\imath\vk^\inc\cdot\vx}\mathrm{e}^{-\imath\omega t}\quad (\vx\in\widetilde{D}_0),
  \label{eq:incident_EH}
\end{eqnarray}
where $\imath$ denote the imaginary unit, i.e. $\imath:=\sqrt{-1}$. Here, the wavenumber vector $\vk^\inc$ is defined as
\begin{eqnarray}
  \vk^\inc := k_0 (\cos\phi\sin\theta, \sin\phi\sin\theta, -\cos\theta)^{\rm T},
  \label{eq:incident}
\end{eqnarray}
where $k_0:=\omega\sqrt{\varepsilon_0\mu_0}\equiv|\vk^\inc|$ denotes the wavenumber in $\widetilde{D}_0$, and has the relationship $\omega\mu_0\bm{b}^\inc=\vk^\inc\times\bm{a}^\inc$. Also, the angles $\theta$ and $\phi$ denote the incident angles from the $-x_3$- and $x_1$-axis, respectively (see Figure~\ref{fig:angles}). In what follows, we will omit the time factor $\mathrm{e}^{-\imath\omega t}$ for the sake of simplicity.

\begin{figure}[H]
  \centering
  \includegraphics[width=.9\textwidth]{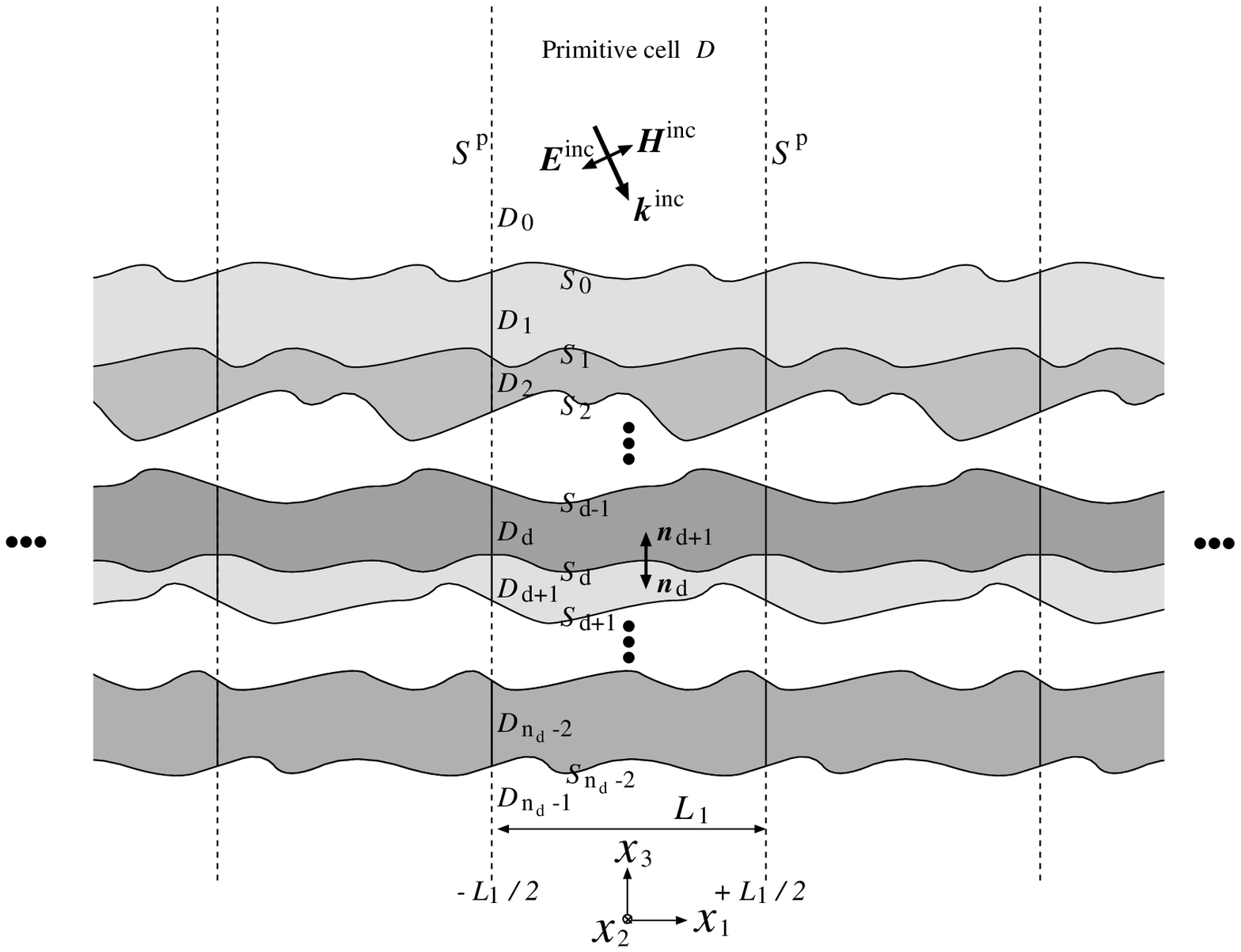}\\[10pt]
  \includegraphics[width=.45\textwidth]{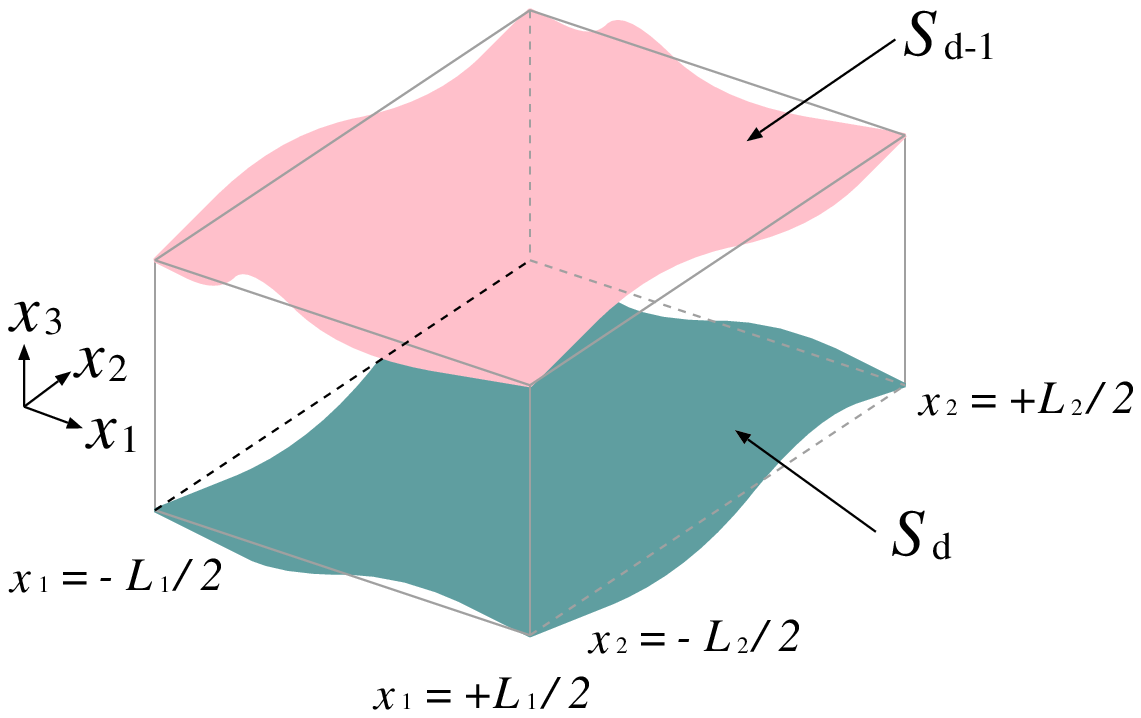}\\
  \caption{Scattering problem in the 3D doubly-periodic multi-layered structure. The lower figure is the bird-view of the domain $D_d$, which is the intersection of the (infinite) domain $\widetilde{D_d}$ and the primitive cell $D$ in (\ref{eq:primitive_cell}).}
  \label{fig:periodicproblem}
\end{figure}

\begin{figure}[H]
  \centering
  \includegraphics[width=.3\textwidth]{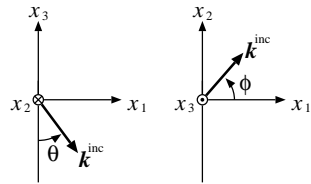}
  \caption{Definition of the angles $\theta$ and $\phi$.}
  \label{fig:angles}
\end{figure}

From the periodicity, we may consider only the following primitive cell $D$:
\begin{eqnarray}
  D := \left[-L_1/2,L_1/2\right]\otimes\left[-L_2/2,L_2/2\right]\otimes\left[-\infty, \infty\right].
  \label{eq:primitive_cell}
\end{eqnarray}
The four side boundaries of $D$ are defined as the periodic boundary $S^\p$, i.e.
\begin{eqnarray}
  S^\p := \left\{\vx \mid \text{$\vx\in\partial D$, $|x_1|=L_1/2$ or $|x_2|=L_2/2$} \right\}.
  \label{eq:Sp}
\end{eqnarray}
Then, the part of $\widetilde{D}_d$ in $D$ is denoted by $D_d$, i.e. $D_d:=\widetilde{D}_d\cap D$, and  the interface (boundary) between $D_d$ and $D_{d+1}$ is denoted by $S_d$, i.e. $S_d := \partial D_d \cap \partial D_{d+1}$. We let $\vn_d$ be the unit outward normal vector of $\partial D_d$.

Then, we may solve the following periodic boundary value problem in the primitive cell $D$:\\
\begin{subequations}
  Maxwell equations in $D_d$:
  \begin{align}
    &\nabla\times\E_d(\vx) = \imath\omega\mu_d\H_d(\vx),\label{eq:maxwell1}\\
    &\nabla\times\H_d(\vx) = -\imath\omega\varepsilon_d\E_d(\vx),\label{eq:maxwell2}\\
    &\text{($\E_d$, $\H_d$: electric and magnetic fields in terms of $D_d$)}\nonumber
  \end{align}
  Boundary conditions on $S_d$ ($=\partial D_d\cap\partial D_{d+1}$):
  \begin{align}
    & \J_d(\vx)=-\J_{d+1}(\vx),\label{eq:qpbBC1}\\
    & \M_d(\vx)=-\M_{d+1}(\vx),\label{eq:qpbBC2}\\
    &\text{($\J_d:=\vn_d\times\H_d$, $\M_d:=\E_d\times\vn_d$: surface electric and magnetic current densities in terms of $\partial D_d$)}\nonumber
  \end{align}
  Quasi-periodic conditions on $\partial D_d\cap S^\p$:
  \begin{align}
    &\E_d({L_1}/{2}, x_2, x_3) = \mathrm{e}^{\imath\beta_1}\E_d(-{L_1}/{2}, x_2, x_3),\quad\H_d({L_1}/{2}, x_2, x_3) = \mathrm{e}^{\imath\beta_1}\H_d(-{L_1}/{2}, x_2, x_3),\label{eq:qpb1}\\
    &\E_d(x_1, {L_2}/{2}, x_3) = \mathrm{e}^{\imath\beta_2}\E_d(x_1, -{L_2}/{2}, x_3),\quad\H_d(x_1, {L_2}/{2}, x_3) = \mathrm{e}^{\imath\beta_2}\H_d(x_1, -{L_2}/{2}, x_3),\label{eq:qpb2}\\
    &\text{($\beta_1$, $\beta_2$ : phase differences defined as $L_1 k^\inc_1$ and $L_2 k^\inc_2$)}\nonumber
  \end{align}
  Radiation conditions for $|x_3|\rightarrow\infty$:
  \begin{align}
    &\E_0(\vx)-\Einc(\vx)=\sum_{\bm{\nu}\in\mathbb{Z}^2}\bm{a}^+_{\bm{\nu}}\e^{\imath\vk_{\bm{\nu}}^+\cdot\vx},\quad\H_0(\vx)-\Hinc(\vx)=\sum_{\bm{\nu}\in\mathbb{Z}^2}\bm{b}^+_{\bm{\nu}}\e^{\imath\vk_{\bm{\nu}}^+\cdot\vx}\quad (x_3 > x_3^{\rm max}),\label{eq:rad1}\\
    &\E_{n_d-1}(\vx)=\sum_{\bm{\nu}\in\mathbb{Z}^2}\bm{a}^-_{\bm{\nu}}\e^{\imath\vk_{\bm{\nu}}^-\cdot\vx},\quad \H_{n_d-1}(\vx)=\sum_{\bm{\nu}\in\mathbb{Z}^2}\bm{b}^-_{\bm{\nu}}\e^{\imath\vk_{\bm{\nu}}^-\cdot\vx}\quad (x_3 < x_3^{\rm min}).\label{eq:rad2}
  \end{align}%
  \label{eq:QPBVP}%
\end{subequations}%
Here, in \autoref{eq:rad1} (respectively, \autoref{eq:rad2}), $x_3^\mathrm{max}$ (respectively, $x_3^\mathrm{min}$) denotes the maximum (respectively, minimum) value of the coordinate $x_3$ in the top interface $S_0$ (respectively, the bottom interface $S_{n_d-2}$)~\cite{arens2010, greengard2014}. Also, $\vk^\pm_{\bm{\nu}}$ is defined by
\begin{eqnarray}
  \vk_{\bm{\nu}}^\pm :=
  \left(
  \begin{array}{c}
    (\beta_1+2\nu_1\pi)/L_1\\
    (\beta_2+2\nu_2\pi)/L_2\\
    \pm\sqrt{k_d^2 - (\beta_1+2\nu_1\pi)^2/L_1^2 - (\beta_2+2\nu_2\pi)^2/L_2^2}
  \end{array}
  \right),
  \label{eq:vkn}
\end{eqnarray}
where $k_d:=\omega\sqrt{\varepsilon_d\mu_d}$ stands for the wavenumber in $D_d$. Also, the vectors $\bm{a}^\pm_{\bm{nu}}$ and $\bm{b}^\pm_{\bm{nu}}$ represent the coefficients of the far-fields.

In this study, we suppose that the third component of $\vk_{\bm{\nu}}^\pm$ in \autoref{eq:vkn} is not zero, which corresponds to the Rayleigh's anomaly and thus prohibits us from calculating the periodic Green's function mentioned below.

\subsection{Boundary integral equations}\label{s:bie}
We solve the periodic boundary value problem \autoref{eq:QPBVP} with the PMCHWT-type boundary integral equations (BIEs)~\cite{poggio1973}, i.e.
\begin{subequations}
  \begin{eqnarray}
     &&\sum_{d=i}^{i+1}\left[\imath\omega\mu_d(\mathscr{L}^\p_d\J_d)(\vx) - (\mathscr{K}^\p_d\M_d)(\vx)\right]_{\rm tan} = \left[\Einc(\vx)\right]_{\textrm{tan},i}\quad(\vx\in S_i),\\
     &&\sum_{d=i}^{i+1}\left[\imath\omega\varepsilon_d(\mathscr{L}^\p_d\M_d)(\vx) + (\mathscr{K}^\p_d\J_d)(\vx)\right]_{\rm tan} = \left[\Hinc(\vx)\right]_{\textrm{tan},i}\quad(\vx\in S_i),
  \end{eqnarray}
  \label{eq:PMCHWT_prd}%
\end{subequations}
for $i=0,\ldots,\nd-2$.\footnote{Eqs.~\autoref{eq:PMCHWT_prd} have not considered the boundary conditions in \autoref{eq:qpbBC1} and \autoref{eq:qpbBC2} yet. In practice, we eliminate either $(\J_d,\M_d)$ or $(\J_{d+1},\M_{d+1})$ on the surface $S_d$ ($=\partial D_d\cap\partial D_{d+1}$) by the boundary conditions and then solve \autoref{eq:PMCHWT_prd} for the remaining variables.} Here, $[\bm{v}]_{\rm tan}$ denotes the tangential component of a vector field $\bm{v}$. Moreover, $[\{\E,\H\}^\inc]_{\textrm{tan},i}$ represents $[\{\E,\H\}^\inc]_{\textrm{tan}}$ if $i=0$ and vanishes otherwise. Also, the following integral operators are defined:
\begin{eqnarray}
  &&(\mathscr{L}^\p_d\vt_d)(\vx) := \int_{\partial D_d \setminus S^\p}\left(1+\frac{1}{k_d^2}\nabla_x\nabla_x\cdot\right)G_d^\p(\vx-\vy) \vt_d(\vy) \dS_y,\label{eq:opL}\\
  &&(\mathscr{K}^\p_d\vt_d)(\vx) := \int_{\partial D_d \setminus S^\p}\vt_d(\vy)\times\nabla_y G_d^\p(\vx-\vy) \dS_y.\label{eq:opK}
\end{eqnarray}
Here, with denoting the fundamental solution for the 3D Helmholtz equation by $G_d(\vx):=\frac{\mathrm{e}^{\imath k_d |\vx|}}{4\pi|\vx|}$, $\Gp_d$ represents the periodic Green's function, by which $\E_d$ and $\H_d$ can satisfy the quasi-periodic conditions and radiation conditions in \autoref{eq:QPBVP}, and has the following formal expression:
\begin{eqnarray}
  \Gp_d(\vx-\vy)
  = \sum_{\bm{\nu}:=(\nu^1, \nu^2)\in\mathbb{Z}^2}\mathrm{e}^{\imath\vk^\inc\cdot\bm{p}^{(\bm{\nu})}} G_d(\vx - (\vy + \bm{p}^{(\bm{\nu})}))
  = \sum_{\bm{\nu}\in\mathbb{Z}^2}\mathrm{e}^{\imath\vk^\inc\cdot\bm{p}^{(\bm{\nu})}} \frac{\mathrm{e}^{\imath k_d |\vx - (\vy + \bm{p}^{(\bm{\nu})})|}}{4\pi|\vx - (\vy + \bm{p}^{(\bm{\nu})})|},
  \label{eq:Gp}
\end{eqnarray}
where $\bm{p}^{(\bm{\nu})}$ denotes the translation vector
\begin{eqnarray*}
  \bm{p}^{(\bm{\nu})} := (L_1\nu_1,\ L_2\nu_2,\ 0)^\mathrm{T}.
\end{eqnarray*}

Since the wavenumber $k_d$ is real by assumption, the infinite series in \autoref{eq:Gp} converges very slowly or often does converge. Therefore, we compute $G_d^\p$ with the Ewald's method~\cite{arens2010}. The details are described in \ref{s:ewald}.

We solve the BIEs in \autoref{eq:PMCHWT_prd} by the Galerkin method. Denoting the weight (testing) vector function as $\bm{w}$, we can obtain the following variational integral equations:
\begin{subequations}
  \begin{eqnarray}
    &&\Bigl\langle \bm{w}(\vx), \sum_{d=i}^{i+1}\left[\imath\omega\mu_d(\mathscr{L}^\p_d\J_d)(\vx) - (\mathscr{K}^\p_d\M_d)(\vx)\right]_{\rm tan} \Bigr\rangle = \Bigl\langle \bm{w}(\vx),\left[\Einc(\vx)\right]_{\textrm{tan},i}\Bigr\rangle,\\
    &&\Bigl\langle \bm{w}(\vx), \sum_{d=i}^{i+1}\left[\imath\omega\varepsilon_d(\mathscr{L}^\p_d\M_d)(\vx) + (\mathscr{K}^\p_d\J_d)(\vx)\right]_{\rm tan}\Bigr\rangle = \Bigl\langle \bm{w}(\vx), \left[\Hinc(\vx)\right]_{\textrm{tan},i}\Bigr\rangle,
  \end{eqnarray}
  \label{eq:mom}%
\end{subequations}
where we let $\langle \bm{a},\bm{b} \rangle:=\int\bm{a}(\vx)\cdot\bm{b}(\vx)\dS_x$.

\subsection{Regularisation} \label{s:regularisation} 

As in the case of the conventional (triangular) RWG~\cite{rao1982} and (square) rooftop basis functions, it is useful to regularise the integral operator $\mathscr{L}^\p_d$ in the BIEs \autoref{eq:mom} in order to reduce its singularity owing to two differential operators, i.e. $\nabla_x\nabla_x$. The regularisation can be performed by moving one differentiation to a surface current density $\vt_d$ (${}=\J_d$ or $\M_d$) and the other to a weight function $\bm{w}$ by using integration by parts. As a result, we will see that the surface current densities and weight function need to satisfy certain quasi-periodic conditions. These conditions imply that we need to use appropriate basis (approximation) and weight functions when we discretise the variational integral equations \autoref{eq:mom}. The discretisation will be investigated in \autoref{s:basis}.

We regularise the second term in the integral operator $\mathscr{L}^\p_d$ in \autoref{eq:opL} in terms of a surface current density $\vt_d$. We now consider an interface $S_d$, which is the lower boundary of $D_d$, and regularise the following integral over $S_d$:
\begin{eqnarray}
  \bm{I}(\vx):=\nabla_x\nabla_x\cdot\int_S \Gp(\vx-\vy) \vt(\vy) \dS_y,
  \label{eq:target_integral}
\end{eqnarray}
where we omit the domain index $d$ from the symbols $S_d$, $\Gp_d$, and $\vt_d$ for simplicity.

\subsubsection{Requirement for vector basis function}\label{s:req_basis}

From the fact that $\vt$ is tangential to $S$ and the Gauss's divergence theorem~\cite{kirsch2014}, \autoref{eq:target_integral} except for $\nabla_x$ can be rewritten as follows:
\begin{eqnarray*}
  \nabla_x\cdot\int_S \Gp(\vx-\vy) \vt(\vy) \dS_y
  &=&\int_S \nabla_x\Gp(\vx-\vy)\cdot\vt(\vy)\ \dS_y\notag\\
  &=&-\int_S \nabla_y\Gp(\vx-\vy)\cdot\vt(\vy)\ \dS_y\notag\\
  &=&\int_S \Gp(\vx-\vy)\sdiv\vt(\vy) \dS_y - \oint_{\partial S} \Gp(\vx-\vy)\vt(\vy)\cdot\bm{\tau}(\vy) \diff\ell_y,
\end{eqnarray*}
where the unit vector $\bm{\tau}$ is tangential to $S$ and normal to $\partial S$. Further, $\sdiv$ denotes the surface divergence~\cite{kirsch2014}. Since $\partial S$ consists of the four integral paths, i.e. $C_1$, $C_2$, $C_3$, and $C_4$ (see \autoref{fig:delS_C}, again), we can rewrite the above path integral as follows:
\begin{eqnarray*}
  \oint_{\partial S} \Gp(\vx-\vy)\vt(\vy)\cdot\bm{\tau}(\vy) \diff\ell_y
   &=& \int_{C_1\cup C_3} \Gp(\vx-\vy)\vt(\vy)\cdot\bm{\tau}(\vy) \diff\ell_y + \int_{C_2\cup C_4} \Gp(\vx-\vy)\vt(\vy)\cdot\bm{\tau}(\vy) \diff\ell_y\nonumber\\
  &=& \int_{C_1} \Gp(\vx-\vy)\left(\vt(\vy)\cdot\bm{\tau}(\vy) - \mathrm{e}^{-\imath\beta_1}\vt(\vy+L_1\bm{e}_1)\cdot\bm{\tau}(\vy)\right) \diff\ell_y\notag\\
  &&+ \int_{C_2} \Gp(\vx-\vy)\left(\vt(\vy)\cdot\bm{\tau}(\vy) - \mathrm{e}^{-\imath\beta_2}\vt(\vy+L_2\bm{e}_2)\cdot\bm{\tau}(\vy)\right) \diff\ell_y,\nonumber
\end{eqnarray*}
where we used the identities $\Gp(\vx-(\vy+L_1\bm{e}_1))=\Gp(\vx-(\vy+\bm{p}^{(1,0)}))=\mathrm{e}^{\vk^\inc\cdot(-\bm{p}^{(1,0)})}\Gp(\vx-\vy)=\mathrm{e}^{-i\beta_1}\Gp(\vx-\vy)$ for any $\vx$ and $\vy$ and $\bm{\tau}_1(\vy+L_1\bm{e})=-\bm{\tau}(\vy)$ for any $\vy$ on $C_1$ in the first term of the most RHS; we used similar identities in the second term. Therefore, in order to eliminate the path integral, we need the following conditions:
\begin{subequations}
  \begin{eqnarray}
    &&\vt(\vx + L_1\bm{e}_1)\cdot\bm{\tau}(\vx) = \mathrm{e}^{\imath\beta_1}\vt(\vx)\cdot\bm{\tau}(\vx)\quad(\vx\in C_1),\label{eq:Pdivconf1}\\
    &&\vt(\vx + L_2\bm{e}_2)\cdot\bm{\tau}(\vx) = \mathrm{e}^{\imath\beta_2}\vt(\vx)\cdot\bm{\tau}(\vx)\quad(\vx\in C_2).\label{eq:Pdivconf2}   
  \end{eqnarray}
  \label{eq:Pdivconf}%
\end{subequations}
These equations represent quasi-periodic conditions for the normal component of the surface current density $\vt$ on the boundary $\partial S$ or the periodic boundary $S^\p$. In \autoref{s:basis}, we approximate $\vt$ with the basis function so that \autoref{eq:Pdivconf} are satisfied.

\begin{figure}[H]
  \centering
  \includegraphics[width=.4\textwidth]{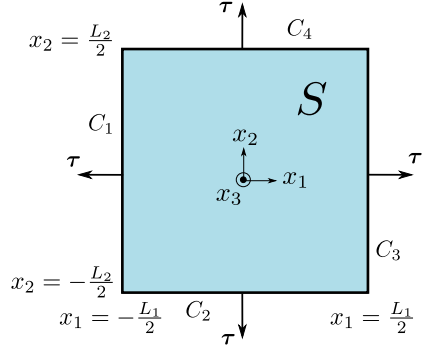}
  \caption{Interface $S$, whose boundary $\partial S$ consists of the four paths $C_1$, $C_2$, $C_3$, and $C_4$. The vector $\bm{\tau}$ denotes the $S$'s tangential vector that is normal to the $\partial S$. The interface $S$ is plane in this figure but not plane in general.}
  \label{fig:delS_C}
\end{figure}

\subsubsection{Requirement for weight function}\label{s:req_weight}

Let us consider a vector weight function $\bm{w}$ that is tangential to $S$. Then, the testing of the underlying vector $\bm{I}$ in \autoref{eq:target_integral}, where the path integral is removed by assuming \autoref{eq:Pdivconf}, can be expressed as follows:
\begin{eqnarray*}
  &&\int_S\bm{w}(\vx)\cdot\bm{I}(\vx) \dS_x 
    = \int_S\bm{w}(\vx)\cdot\left(\nabla_x\int_S\Gp(\vx-\vy)\sdiv\vt(\vy) \dS_y\right) \dS_x\nonumber\\
  &&= \oint_S \bm{w}(\vx)\cdot\bm{\tau}(\vx) \left(\int_S \Gp(\vx-\vy)\sdiv\vt(\vy) \dS_y\right) \diff\ell_x
    - \int_S \sdiv\bm{w}(\vx) \int_S\Gp(\vx-\vy)\sdiv\vt(\vy) \dS_y \dS_x.
\end{eqnarray*}
Similarly to the previous subsection, the path integral vanishes if the following conditions are met:
\begin{subequations}
  \begin{eqnarray}
    &\bm{w}(\vx + L_1\bm{e}_1)\cdot\bm{\tau}(\vx) = \mathrm{e}^{-\imath\beta_1}\bm{w}(\vx)\cdot\bm{\tau}(\vx)\quad(\bm{x}\in C_1),\\
    &\bm{w}(\vx + L_2\bm{e}_2)\cdot\bm{\tau}(\vx) = \mathrm{e}^{-\imath\beta_2}\bm{w}(\vx)\cdot\bm{\tau}(\vx)\quad(\bm{x}\in C_2).   
  \end{eqnarray}
  \label{eq:Pdivconf_w}%
\end{subequations}
Namely, as far as we choose a weight function that satisfies the quasi-periodic condition regarding the inverse phase difference (i.e. $e^{-\imath\beta_1}$ and $e^{-\imath\beta_2}$), we can move the underlying differential operator $\nabla_x$ to the weight function. 

The requirements in (\ref{eq:Pdivconf}) and (\ref{eq:Pdivconf_w}) were already mentioned by Otani et al.~\cite[Section 2.2.3]{otani2008} and rigorously studied by Hu et al.~\cite{hu2011} in terms of the RWG basis.

\section{Doubly-periodic surface}\label{s:surface}

We desire to express each doubly-periodic interface (open surface) with the B-spline function. To this end, after defining the B-spline function and curve in \autoref{s:surf_def}, we first show an algorithm to build a periodic curve with the B-spline curve (\autoref{s:surf_pBcurve}). Successively, we construct an open surface as the tensor product of two periodic curves (\autoref{s:surf_pBsurface}).

\subsection{Definitions}\label{s:surf_def}

Let $B_0^p,\ldots,B_{n-1}^p$ be $n$ ($\ge 1$) B-spline functions of degree $p$ ($\ge 0$) and $T:=\{t_0,\ldots,t_{n+p}\}$ be the knot vector, where the knots $t_0$, $\ldots$, $t_{n+p}$ satisfy $t_0\le t_1\le \ldots \le t_{n+p}$ in general. In this study, we compute the B-spline function according to the following Cox-de Boor recursion formula~\cite{piegl2012}:
\begin{subequations}
  \begin{eqnarray}
    &&B_i^0(t) =
       \begin{cases}
         1 & (t_i \leq t < t_{i+1})\\
         0 & (\text{otherwise})
       \end{cases},\\
    &&B_i^p(t) = \frac{t - t_i}{t_{i+p} - t_i}B_i^{p-1}(t) + \frac{t_{i+p+1} - t}{t_{i+p+1}-t_{i+1}}B_{i+1}^{p-1}(t)\quad(p\ge 1).
  \end{eqnarray}
  \label{eq:CoxdeBoor}%
\end{subequations}
It should be noted that the support of $B_i^p$ is $[t_i,t_{i+p+1}]$.

Here, the partition of unity holds in the domain $[t_p,t_n]$, i.e.
\begin{eqnarray}
  \sum_{i=0}^{n-1} B_i^p(t)\equiv 1 \quad(t\in[t_p,t_n]).
  \label{eq:pou}
\end{eqnarray}
We assume that the domain $[t_p,t_n]$ is non-vanishing, that is,
\begin{eqnarray}
  n > p.
  \label{eq:n>p}
\end{eqnarray}
Then, we define a B-spline curve $C\subset \bbbr^2$ as a set of the points $\bm{x}\in\bbbr^2$ such as
\begin{eqnarray}
  \bm{x}(t)=\begin{pmatrix}x(t)\\y(t)\end{pmatrix}:=\sum_{i=0}^{n-1} B_i^p(t)\bm{p}_i\quad(t\in[t_p,t_n]),
  \label{eq:Bcurve}
\end{eqnarray}
where $\bm{p}_i\equiv(x_i,y_i)^{\rm T}\in\bbbr^2$ denotes the given $i$-th control point.

\subsection{Periodic B-spline curve}\label{s:surf_pBcurve}

Let us consider a non-self-intersecting B-spline curve (on the $xy$-plane) that connects a point on the line $x=-\frac{L}{2}$ with another on $x=\frac{L}{2}$, where $L$ will denote the period in the $x$ direction. In particular, we request that the knots $t_p$ and $t_n$ (i.e. the ends of the parametric coordinate $t$) correspond to $x=-\frac{L}{2}$ and $\frac{L}{2}$ (i.e. those of the physical coordinate $x$), respectively, that is,
\begin{eqnarray}
  x(t_p)=-\frac{L}{2},\quad x(t_n)=\frac{L}{2}.
  \label{eq:pB1}
\end{eqnarray}
Moreover, in order to let the curve be periodic (of period $L$) and continuously differentiable, we suppose 
\begin{eqnarray}
  \frac{\diff^k y}{\diff t^k}(t_p)=\frac{\diff^k y}{\diff t^k}(t_n)\quad(k=0,\ldots,p-1).
  \label{eq:pB2}
\end{eqnarray}

\autoref{algo:pBcurve} can construct a B-spline curve that satisfies both \autoref{eq:pB1} and \autoref{eq:pB2}. The mathematical justification of this algorithm is described in \ref{s:proof_surface}, in particular, Theorem~\ref{theo:pBcurve}. We will term such a B-spline curve a \textit{periodic B-spline curve} hereafter.

\begin{algorithm}[H]
  \caption{Construction of a periodic B-spline curve, which satisfies \autoref{eq:pB1} and \autoref{eq:pB2}.}
  \label{algo:pBcurve}
  \begin{algorithmic}[1]
    
    \STATE Give uniform knots, i.e.
    \begin{eqnarray}
      t_i:=\frac{i}{n+p}\quad(i=0,\ldots,n+p).
      \label{eq:pBcurve_cond1}
    \end{eqnarray}
    \STATE Give the first $n-p$ vertical coordinates of the control points, i.e. $y_0,\ldots,y_{n-p-1}$, arbitrarily according to the desired shape of the curve. The remaining $p$ vertical coordinates are determined as
    \begin{eqnarray}
      y_{n-p+i}=y_i \quad(i=0,\ldots,p-1).
      \label{eq:pBcurve_cond2}
    \end{eqnarray}
    \STATE Then, the horizontal coordinates of all the control points are given by
    \begin{eqnarray}
      x_i=-\frac{L}{2}-\frac{L(p-1)}{2(n-p)}+\frac{iL}{n-p}\quad(i=0,\ldots,n-1).
      \label{eq:pBcurve_cond3}
    \end{eqnarray}
    
  \end{algorithmic}
\end{algorithm}

\autoref{fig:test-periodic-bspline} shows an example of a periodic B-spline curve that satisfies \autoref{eq:pB1} and \autoref{eq:pB2}. To generate the curve, we first give $L=0.8$, $n=5$, $p=2$, $y_0=1$, $y_1=3$, and $y_2=-2$. Next, \autoref{eq:pBcurve_cond1} determines the knots as $t_0=0$, $t_1=\frac{1}{7}$, $t_2=\frac{2}{7}$, $\ldots$, and $t_7=1$. Then, \autoref{eq:pBcurve_cond2} determines the remaining vertical coordinates, that is, $y_3=1$ and $y_4=3$. Last, \autoref{eq:pBcurve_cond3} gives the horizontal coordinates of all the $n$ control points as $x_0=-0.5333333$, $x_1=-0.2666667$, $x_2=0.000000$, $x_3=0.2666667$, and $x_4=0.5333333$. Then, we can obtain $x(t_p)=-0.4000000$, $x(t_n)=0.4000000$, $y(t_p)=y(t_n)=2.000000$, and $\frac{\diff y(t_p)}{\diff t}=\frac{\diff y(t_n)}{\diff t}=14.00000$. Therefore, we can confirm that \autoref{eq:pB1} and \autoref{eq:pB2} hold.

\begin{figure}[H]
  \centering
  \includegraphics[width=.5\textwidth]{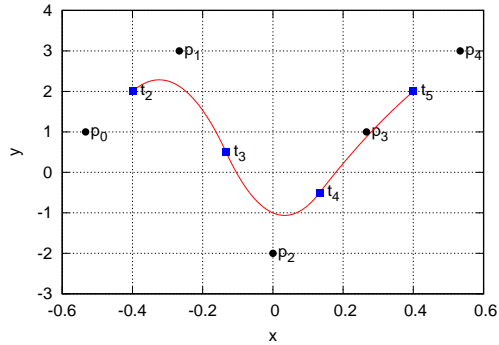}
  \caption{Example of a periodic B-spline curve generated by \autoref{algo:pBcurve}, where $L=0.8$, $n=5$, $p=2$, $y_0=1$, $y_1=3$, and $y_2=-2$ are given. The circles ($\bullet$) denote the control points $\bm{p}_i=(x_i,y_i)^{\rm T}$, where $i=0,\ldots,4$, and the squares (\blue{$\sqbullet$}) denote the points that correspond to $t=t_i$, where $i=0,\ldots,4$.}
  \label{fig:test-periodic-bspline}
\end{figure}

We remark the properties of the periodic B-spline curve:
\begin{remark}{Properties of periodic B-spline curve}
 
  \begin{enumerate}
 
  \item The assumption of uniform knots in \autoref{eq:pBcurve_cond1} does not allow to generate any corner on the curve.
    
  \item The assumption of horizontally uniform control points in \autoref{eq:pBcurve_cond3} restricts the shape of the curve: the $x$-component of the tangential vector of the curve is always positive.
    
  \end{enumerate}
  
  \label{theo:pBcurve_prop}
  
 \end{remark}

We note that the present algorithm to construct a periodic B-spline curve is an extension of the algorithm by Shimba et al.~\cite{shimba2015}. The authors considered the case of $p=2$ only, while we consider the general case of $p$.

\subsection{Periodic B-spline surface}\label{s:surf_pBsurface}

We define a \textit{periodic B-spline surface} by the tensor product of two periodic B-spline curves, each of which is generated by Algorithm~\ref{algo:pBcurve}. Then, the point $\vx$ ($\in\bbbr^3$) on a periodic B-spline surface can be written as follows:
\begin{eqnarray}
  \vx(t_1,t_2) = \sum_{i=0}^{n_1-1}\sum_{j=0}^{n_2-1}B_i^{p_1}(t_1)B_j^{p_2}(t_2)\bm{p}_{ij}\quad((t_1,t_2)\in [t_{1,p_1},t_{1,n_1}]\otimes[t_{1,p_2},t_{2,n_2}]),
  \label{eq:pBsurf_x}
\end{eqnarray}
where $n_h$ and $p_h$ denotes the number of control points and the degree of B-spline function, respectively, regarding the coordinate $t_h$, where $h=1$ and $2$. Correspondingly, $t_{h,k}$ denotes the $k$-th knot regarding the coordinate $t_h$, where $0\le k\le n_h+p_h$. Also, $\bm{p}_{ij}$ ($\in\bbbr^3$) denotes the $(i,j)$-th control point.

We note that a periodic B-spline surface owns the following properties:
\begin{remark}{Properties of periodic B-spline surface}
  \label{theo:pBsurfprop}

  \begin{enumerate}
    
  \item The boundaries of $t_1=t_{1,p_1}$, $t_1=t_{1,n_1}$, $t_2=t_{2,p_2}$, and $t_2=t_{2,n_2}$ of the parametric coordinates correspond to those of $x_1=-\frac{L_1}{2}$, $x_1=\frac{L_1}{2}$, $x_2=-\frac{L_2}{2}$ and $x_2=\frac{L_2}{2}$ of the physical coordinates, respectively.

  \item The assumption of the uniform knots in \autoref{eq:pBcurve_cond1} does not allow to generate any corners or edges on the periodic B-spline surface (recall the first item of Remark~\ref{theo:pBcurve_prop}). In this case, we can regard a piece of rectangle region $[t_{1,i},t_{1,i+1}]\otimes[t_{2,j},t_{2,j+1}]$ (where $p_1\leq i < n_1$ and $p_2\leq j < n_2$) as a boundary element. Such an element is called a B\'ezier element in the context of the isogeometric analysis. Since there are $(n_1-p_1)(n_2-p_2)$ elements on a surface, we will denote the $i$-th element by $E_i$, where $1\leq i \leq (n_1-p_1)(n_2-p_2)$.

  \item The surface $\vx$ and its derivatives are continuous across the periodic boundary $S^\p$, i.e.
    \begin{subequations}
      \begin{eqnarray}
        &&\frac{\partial^k \vx}{\partial t_1^k}(t_{1,p_1},t_2)=\frac{\partial^k \vx}{\partial t_1^k}(t_{1,n_1},t_2)\quad(k=0,\ldots,p_1-1),\\
        &&\frac{\partial^k \vx}{\partial t_2^k}(t_1,t_{2,p_2})=\frac{\partial^k \vx}{\partial t_2^k}(t_1,t_{2,n_2})\quad(k=0,\ldots,p_2-1).
      \end{eqnarray}%
      \label{eq:pBsurfprop2}%
    \end{subequations}

  \item As we will see in the construction of the basis function (\autoref{s:basis}), we will actually request that $\vx$ and its first order derivative $\frac{\partial\vx}{\partial t}$ is continuous beyond the periodic boundary. The continuity for the higher order derivatives (i.e. $\frac{\partial^2\vx}{\partial t^2}$, $\frac{\partial^3\vx}{\partial t^3}$, $\ldots$) is unnecessary for this purpose, but would be helpful to compute the tangential derivative of the basis function in some applications such as the shape optimisation.
    
  \end{enumerate}
  
\end{remark}

We will construct each of the interfaces, viz. $S_0$, $\ldots$, $S_{n_d-2}$, by a periodic B-spline surface.

\section{Quasi-periodic basis and weight functions}\label{s:basis}

First, in \autoref{s:buffa}, we will introduce the vector basis function proposed by Buffa et al.~\cite{buffa2010} in order to discretise a surface current density on a surface. This vector basis function is said to be \textit{compatible} in the sense that it obeys the finite dimensional de Rahm diagram~\cite{simpson2018}. Next, in \autoref{s:qpv_basis}, we will modify the Buffa's vector basis function so that it can satisfy the quasi-periodic conditions in \autoref{eq:Pdivconf}, which is necessary to regularise the variational integral equations in \autoref{eq:mom}. Correspondingly, we will construct the vector weight function that can satisfy another regularisation condition in \autoref{eq:Pdivconf_w}. Finally, in \autoref{s:qpv_weight}, we will mention how to determine the B-spline parameters of the vector basis function by considering the assumptions on the parameters.

\subsection{Buffa's vector basis function~\cite{buffa2010}}\label{s:buffa}

We introduce the vector basis function proposed by Buffa et al.~\cite{buffa2010}. Let us consider a rectangular and smooth surface, denoted by $S$, which is parameterised with the coordinates $u_1$ and $u_2$. Then, the following vector functions $\bm{V}_{1,i,j}$ and $\bm{V}_{2,i,j}$ based on the B-spline functions can be the basis of a surface electric or magnetic current density at $\vx$ ($=\vx(u_1,u_2)\in S$):
\begin{subequations}
  \begin{eqnarray}
    &&\bm{V}_{1,i,j}(\vx):=\frac{1}{J}B_i^{q_1}(u_1)B_j^{q_2-1}(u_2)\frac{\partial \vx}{\partial  u_1}\quad\text{($0\leq i < m_1$, $0\leq j < m_2-1$)},\label{eq:compatibleB1}\\
    &&\bm{V}_{2,i,j}(\vx):=\frac{1}{J}B_i^{q_1-1}(u_1)B_j^{q_2}(u_2)\frac{\partial \vx}{\partial  u_2}\quad\text{($0\leq i < m_1-1$, $0\leq j < m_2$)},\label{eq:compatibleB2}%
  \end{eqnarray}%
  \label{eq:compatibleB}%
\end{subequations}
where $J=J(u_1,u_2)$ denotes the Jacobian, i.e. $\left|\frac{\partial \vx}{\partial  u_1}\times\frac{\partial \vx}{\partial  u_2}\right|$. Also, $U_1 := \{u_{1,0},\ldots,u_{1,m_1+q_1}\}$ and $U_2 := \{u_{2,0},\ldots,u_{2,m_2+q_2}\}$ denote the knot vectors of $B_i^{q_1}(u_1)$ and $B_j^{q_2}(u_2)$, respectively. Then, $U'_1 := \{u_{1,1},\ldots,u_{1,m_1+q_1-1}\}$ and $U'_2 := \{u_{2,1},\ldots,u_{2,m_2+q_2-1}\}$ represent the knot vectors of $B_i^{q_1-1}(u_1)$ and $B_j^{q_2-1}(u_2)$, respectively\footnote{We use the symbols of $m_h$ (number of B-spline functions), $q_h$ (degree), $u_{h,\cdot}$ (knots), and $U_h$ (knot vector) for the underlying vector basis function in order to distinguish them from the symbols of $n_h$, $p_h$, $t_{h,\cdot}$, and $T_h$ for the periodic B-spline surface, where $h=1,2$.}.

For the sake of simplicity, we use the notations
\begin{subequations}
  \begin{eqnarray}
    &&\mb_1:=m_1,\quad \mb_2=m_2-1,\quad \mb_3:=m_1-1,\quad \mb_4:=m_2,\\
    &&\qb_1 := q_1,\quad \qb_2 := q_2-1,\quad \qb_3 := q_1-1,\quad \qb_4 :=q_2.
  \end{eqnarray}
  \label{eq:mb_qb}
\end{subequations}
Then, the vector basis function $\bm{V}_{h,i,j}$ (where $h=1,2$) in \autoref{eq:compatibleB} can be expressed as follows:
\begin{eqnarray}
  \bm{V}_{h,i,j}(\vx)=\frac{1}{J} B_i^{\qb_{2h-1}}(u_1)B_j^{\qb_{2h}}(u_2)\frac{\partial \vx}{\partial u_h} \quad \text{($0\leq i < \mb_{2h-1}$, $0\leq j < \mb_{2h}$)}.
  \label{eq:Nhij}
\end{eqnarray}

Correspondingly, the surface divergence can be expressed as follows~\cite{arnoldus2006}:
\begin{eqnarray}
  \sdiv\bm{V}_{h,i,j}(\vx) = \frac{1}{J}\frac{\partial}{\partial u_h}\left( B_i^{\qb_{2h-1} }(u_1) B_j^{\qb_{2h}}(u_2)\right).
\end{eqnarray}

\subsection{Quasi-periodic vector basis function}\label{s:qpv_basis}

We propose a vector basis function by modifying the compatible vector basis function $\bm{V}_{h,i,j}$ in \autoref{eq:Nhij} so that it can satisfy the requisite in \autoref{eq:Pdivconf} or the quasi-periodicity of the tangential component of the surface current on an interface, denoted by $S$.

We assume that $S$ is constructed as a periodic B-spline surface mentioned in Section~\ref{s:surf_pBcurve}. Then, the B-spline functions used for $S$ are determined by a set of parameters, i.e. $n_1$, $n_2$, $p_1$, $p_2$, $t_{1,i}$, and $t_{2,j}$, whereas those used for $\bm{V}_{h,i,j}$ are determined by another set, i.e. $m_1$, $m_2$, $q_1$, $q_2$, $u_{1,i}$, and $u_{2,j}$.

In order to construct a vector basis function that satisfies \autoref{eq:Pdivconf}, we need to give some constraints among the parameters in the above two sets. First, the both B-spline functions should be handled as the functions of the common surface parameters, say $(t_1,t_2)$. Then, the basis function's domain of definition, i.e. $[u_{1,q_1},u_{1,m_1}]\otimes[u_{2,q_2},u_{2,m_2}]$, should be identical to the surface's one, i.e. $[t_{1,p_1},t_{1,n_1}]\otimes[t_{2,p_2},t_{2,n_2}]$, which corresponds to the physical domain $[-L_1/2,L_1/2]\otimes[-L_2/2,L_2/2]$ due to the first property in Remark~\ref{theo:pBsurfprop}. To this end, we assume that the knots of $\bm{V}_{h,i,j}$ satisfy the following relationships with those of $S$:
\begin{eqnarray}
  &u_{1,q_1}=t_{1,p_1},\quad u_{1,m_1}=t_{1,n_1},\quad u_{2,q_2}=t_{2,p_2},\quad u_{2,m_2}=t_{2,n_2}\nonumber\\
  \Leftrightarrow\quad &u_{1,\qb_1}=t_{1,p_1},\quad u_{1,\mb_1}=t_{1,n_1},\quad u_{2,\qb_4}=t_{2,p_2},\quad u_{2,\mb_4}=t_{2,n_2}.
  \label{eq:u=t}
\end{eqnarray}
Second, we assume that there are sufficient numbers of knots, i.e.
\begin{subequations}
  \begin{eqnarray}
    \mb_1\ge 2\qb_1,\label{eq:alotofknots1}\\
    \mb_4\ge 2\qb_4.\label{eq:alotofknots2}
  \end{eqnarray}%
  \label{eq:alotofknots}%
\end{subequations}
Third and last, we assume that the knots near the both ends satisfy the following conditions:
\begin{subequations}
  \begin{eqnarray}
    &&\Delta u_{1,i} = \Delta u_{1,i+\mb_1-\qb_1}\quad\text{($i=0,\ldots,2\qb_1-1$)},\label{eq:pbasis_knot1}\\
    &&\Delta u_{2,j} = \Delta u_{2,j+\mb_4-\qb_4}\quad\text{($j=0,\ldots,2\qb_4-1$)},\label{eq:pbasis_knot2}
  \end{eqnarray}
  \label{eq:pbasis_knot}%
\end{subequations}
where $\Delta u_{h,i}:=u_{h,i+1}-u_{h,i}$. 

Under these assumptions in \autoref{eq:u=t}, \autoref{eq:alotofknots}, and \autoref{eq:pbasis_knot}, we can prove that the vector basis function
\begin{eqnarray}
  \Mp_{h,i,j}(\vx):=
  \begin{cases}
    \bm{V}_{1,i,j}(\vx) + \e^{\imath\beta_1}\bm{V}_{1,i+\mb_1-\qb_1,j}(\vx) & (h=1,\ 0 \leq i < \qb_1,\ 0 \leq j < \mb_2),\\
    \bm{V}_{2,i,j}(\vx) + \e^{\imath\beta_2}\bm{V}_{2,i,j+\mb_4-\qb_4}(\vx) & (h=2,\ 0 \leq i < \mb_3,\ 0 \leq j < \qb_4),\\
    \bm{V}_{h,i,j}(\vx) & \text{(otherwise)}
  \end{cases}
                          \label{eq:pMdiv}
\end{eqnarray}
can satisfy the quasi-periodic conditions in \autoref{eq:Pdivconf} for $(t_1,t_2)\in[t_{1,p_1},t_{1,n_1}]\otimes[t_{2,p_2},t_{2,n_2}]$. This is proven in \ref{s:proof_Mp}.

\autoref{fig:pBsp} is helpful to understand the construction of $\Mp_{h,i,j}$ in \autoref{eq:pMdiv} intuitively. Since \autoref{eq:pbasis_knot1} and \autoref{eq:pbasis_knot2} are essentially the same as \autoref{eq:translation_assume},  the B-spline functions consisting of $\bm{V}_{h,i,j}$ in \autoref{eq:Nhij} look like the B-spline functions in the figure. When we focus on the B-spline functions for $u_1$ only, $\Mp_{h,i,j}$ in the first case of \autoref{eq:pMdiv} are constructed as the sum of a B-spline function coloured in red, blue, or green at the LHS in the figure and the B-spline function in the same colour at the RHS, where the phase difference $\e^{\imath\beta_1}$ is multiplied to the latter B-spline function in order to satisfy the quasi-periodic condition. Meanwhile, the B-spline functions coloured in black in the figure are not combined to any others; these correspond to the third case of \autoref{eq:pMdiv}.

It should be noted that Shimba et al.~\cite{shimba2015} constructed a (scalar) B-spline basis function that satisfies the quasi-periodic condition. This is essentially the same as \autoref{eq:pMdiv} but for the scalar or 2D Helmholtz problems. Therefore, we emphasise that the quasi-periodic vector basis function $\Mp_{h,i,j}$ in \autoref{eq:pMdiv} is new.

We also note that, in the particular case of $\qb_1=\qb_4=1$, the vector basis function $\bm{V}_{h,i,j}$ in \autoref{eq:Nhij} is the same as the periodic RWG basis function proposed by Hu et al.~\cite{hu2011}, although these bases are different in shape. Following the notations in \cite{hu2011}, we can express the basis function $\Mp_{h,i,j}$ in the first case of \autoref{eq:pMdiv}, for example, in the following split form:
\begin{eqnarray*}
  \bm{M}^\p_{h,i,j}(\vx) = \bm{M}^{\p-}_{h,i,j}(\vx) + \bm{M}^{\p+}_{h,i,j}(\vx),
\end{eqnarray*}
where $\bm{M}^{\p-}_{h,i,j}(\vx) := \bm{V}_{1,i,j}(\vx)$ and $\bm{M}^{\p+}_{h,i,j}(\vx) := \e^{\imath\beta_1}\bm{V}_{1,i+\mb_1-\qb_1,j}(\vx)$.

In contrast to \autoref{eq:pMdiv}, we may consider not only the quasi-periodicity for the normal component (i.e. $\bm{\tau}$-direction) but also that for the tangential component. This is because, since the surface is smooth as mentioned in Remark~\ref{theo:pBsurfprop}, the surface current densities $\J$ and $\M$ are continuous on the periodic boundary $S^\p$. In this case, similarly to $\Mp_{h,i,j}$, we can say that the following vector basis function $\Np_{h,i,j}$ satisfies the quasi-periodicity for both normal and tangential components:
\begin{eqnarray}
  \Np_{h,i,j}(\vx) :=
  \begin{cases}
    &\bm{V}_{h,i,j}(\vx) + \e^{\imath\beta_1}\bm{V}_{h,i+\mb_{2h-1}-\qb_{2h-1},j}(\vx)\\
    &\qquad(h=1,2,\ 0 \leq i < \qb_{2h-1},\ \qb_{2h} \leq j < \mb_{2h}-\qb_{2h}),\\
    &\bm{V}_{h,i,j}(\vx) + \e^{\imath\beta_2}\bm{V}_{h,i,j+\mb_{2h}-\qb_{2h}}(\vx)\\
    &\qquad(h=1,2,\ \qb_{2h-1} \leq i < \mb_{2h-1}-\qb_{2h-1},\ 0 \leq j < \qb_{2h}),\\
    &\bm{V}_{h,i,j}(\vx) + \e^{\imath\beta_1}\bm{V}_{h,i+\mb_{2h-1}-\qb_{2h-1},j}(\vx) + \e^{\imath\beta_2}\bm{V}_{h,i,j+\mb_{2h}-\qb_{2h}}(\vx)\\
    &\quad+ \e^{\imath\beta_1}\e^{\imath\beta_2}\bm{V}_{h,i+\mb_{2h-1}-\qb_{2h-1},j+\mb_{2h}-\qb_{2h}}(\vx)\\
    &\qquad(h=1,2,\ 0 \leq i < \qb_{2h},\ 0 \leq j < \qb_{2h-1}),\\
    &\bm{V}_{h,i,j}(\vx)\\
    &\qquad\text{(otherwise)}.
  \end{cases}
  \label{eq:pNdiv}
\end{eqnarray}

\autoref{fig:Np} visualises $\Np_{h,i,j}$ for four cases of $(h,i,j)=(1,1,2)$, $(2,0,1)$, $(1,1,0)$, and $(1,3,3)$, which correspond to the first to fourth case in the RHS of \autoref{eq:pNdiv}, respectively.

\begin{figure}[H]
  \centering
  \begin{tabular}{cccc}
    \includegraphics[width=.22\textwidth]{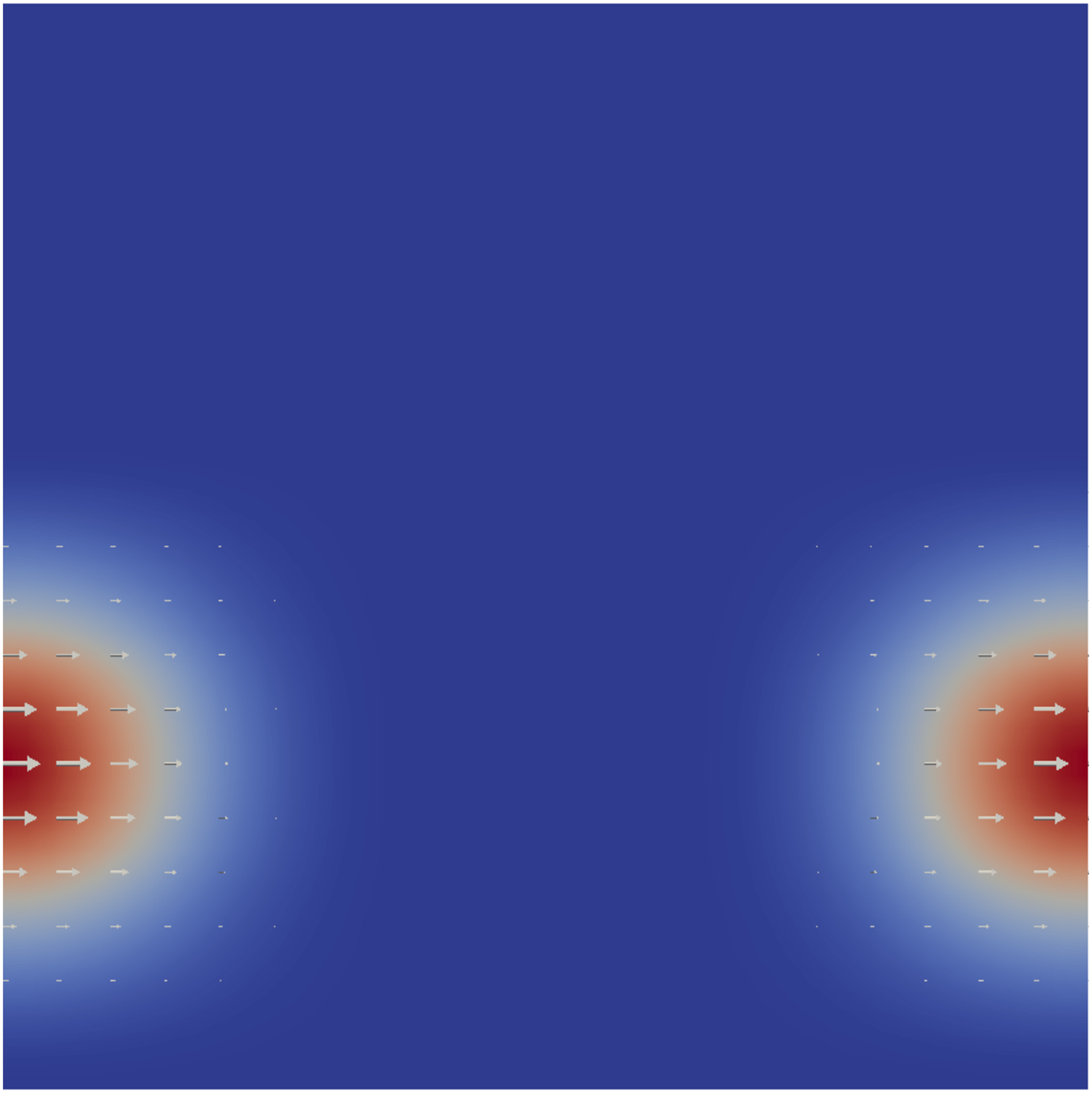}
      &\includegraphics[width=.22\textwidth]{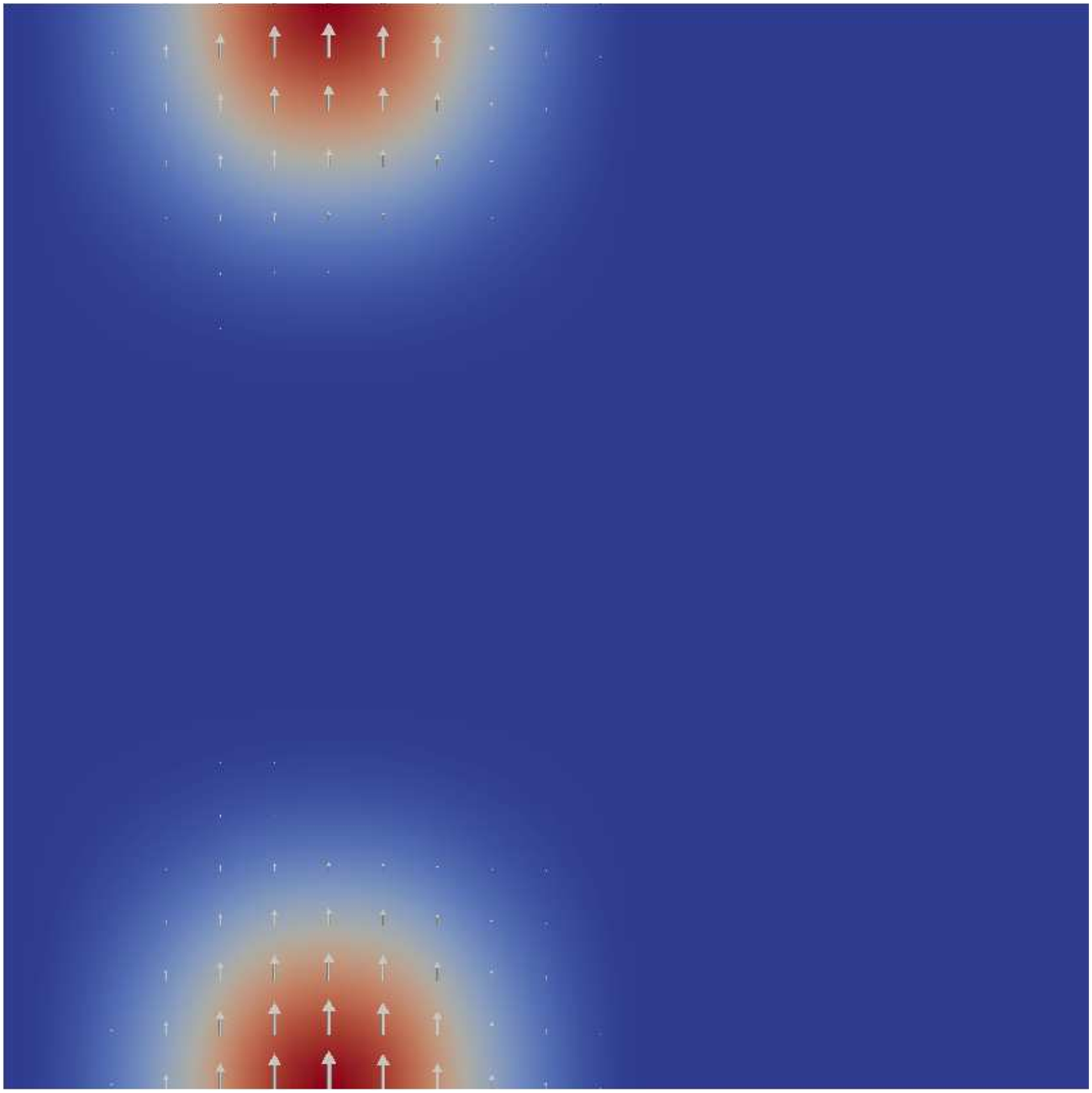}
      &\includegraphics[width=.22\textwidth]{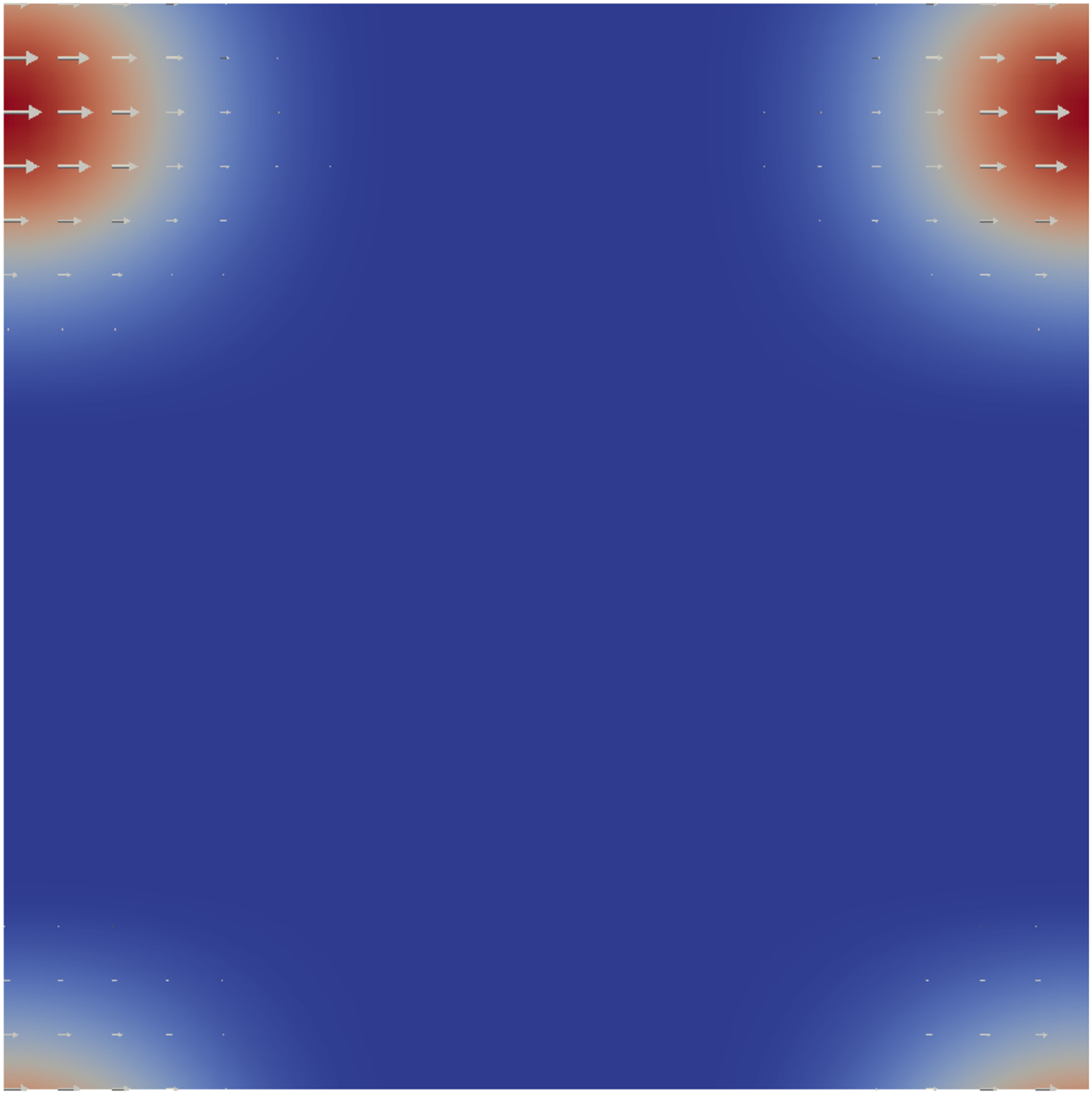}
      &\includegraphics[width=.22\textwidth]{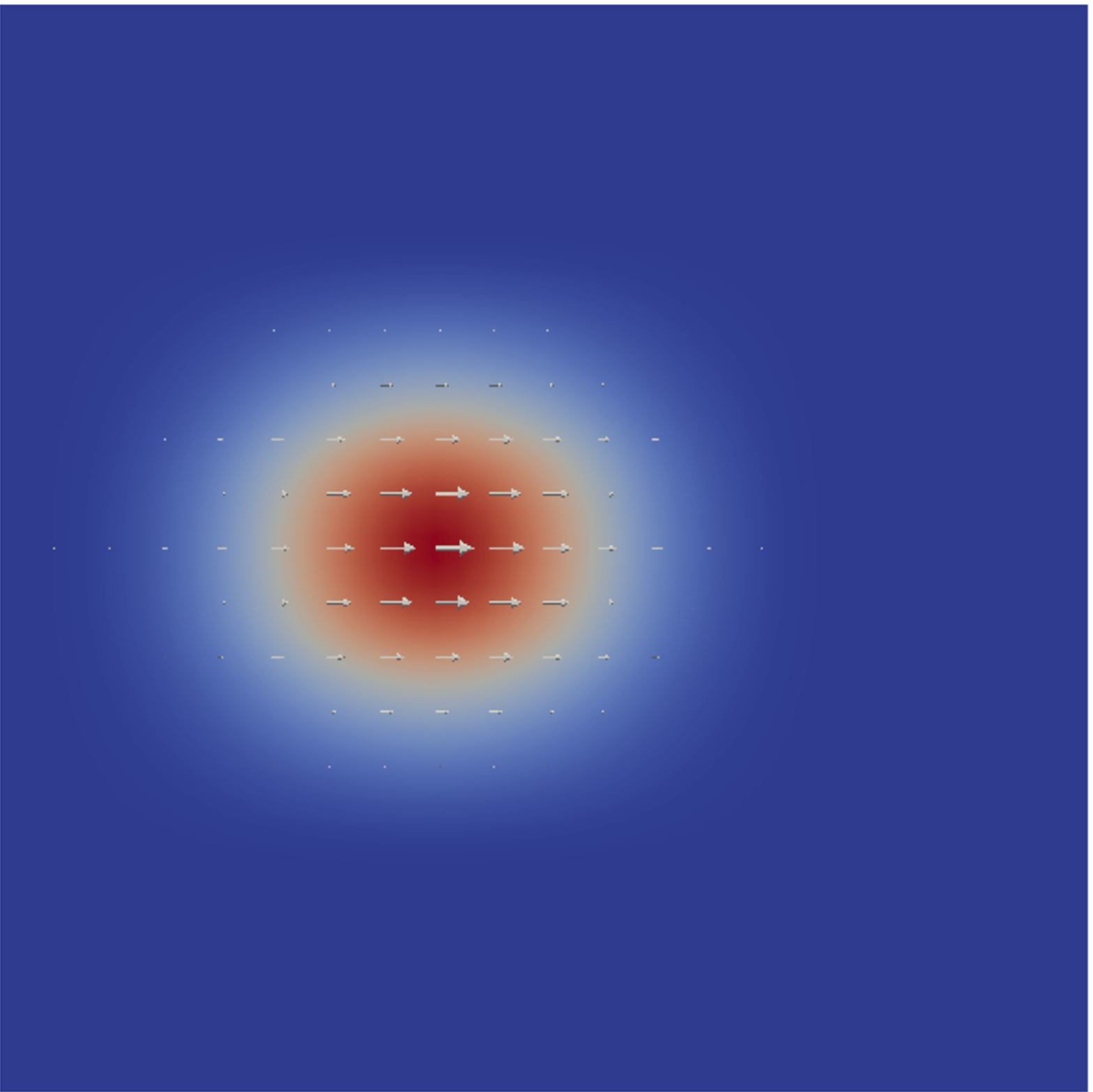}\\
    $\Np_{1,1,2}$ & $\Np_{2,0,1}$ & $\Np_{1,1,0}$ & $\Np_{1,3,3}$\\
  \end{tabular}
  \caption{Vector plot of $\Np_{h,i,j}$ in the case of $q_1=q_2=3$ (thus, $\qb_1=\qb_4=3$ and $\qb_2=\qb_3=2$) on a plane periodic B-spline surface. The colour shows the relative magnitude; blue and red correspond to $0$ and the maximum value of the magnitude $|\Np_{h,i,j}|$, respectively. The present surface is built with the parameters $p_1=p_2=1$, $n_1=n_2=6$, and $L_1=L_2=1$. Then, $m_1$ and $m_2$ are determined to $8$ according to Algorithm~\ref{algo:knots} (then, $\mb_1=\mb_4=8$ and $\mb_2=\mb_3=7$ follow).}
  \label{fig:Np}
\end{figure}

We will use $\Np$ as the vector basis function of a surface current density $\vt$ because a periodic B-spline surface is smooth beyond the periodic boundary. In this case, $\vt$ can be approximated as follows:
\begin{eqnarray}
  \vt(\vx)\approx\sum_{h=1}^2\sum_{i=0}^{\mb_{2h-1}-\qb_{2h-1}-1}\sum_{j=0}^{\mb_{2h}-\qb_{2h}-1} \vtcoef_{h,i,j} \Np_{h,i,j}(\vx) =\sum_{a=1}^{n_a} \vtcoef_{a} \Np_{a}(\vx),
  \label{eq:approx_current}
\end{eqnarray}
where $\vtcoef_{h,i,j}$ represents the (unknown) coefficient. Also, for brevity, we combined the indices $h$, $i$ and $j$ into a single index $a$, with defining $n_a:=(\mb_1-\qb_1)(\mb_2-\qb_2)+(\mb_3-\qb_3)(\mb_4-\qb_4)$. Similarly, we simply denote $\Mp_{h,i,j}$ as $\Mp_a$.

In \autoref{s:num_ex2_N_vs_M}, we will compare $\Np$ with $\Mp$ numerically.

\subsection{Quasi-periodic weight function}\label{s:qpv_weight}

As mentioned in \autoref{s:req_weight}, a vector weight function needs to satisfy the quasi-periodic conditions in \autoref{eq:Pdivconf_w}, where the phase difference is not $\imath\beta_h$ but $-\imath\beta_h$. Therefore, analogously to the basis functions, we may use the complex conjugate of $\Np_a$ in \autoref{eq:pNdiv}, that is,
\begin{eqnarray}
  \bm{w}_a^\p(\vx) = \overline{\Np_a(\vx)},
  \label{eq:weight_function}                             
\end{eqnarray}
where $\overline{(\cdot)}$ denotes the complex conjugate.

\subsection{Determining the knots for vector basis functions}\label{s:vector_knot}

The parameters of B-spline functions for both surface and vector basis function must be predefined. Regarding the surface, we determine the values of $n_1$, $n_2$, $p_1$, $p_2$, $t_{1,i}$, and $t_{2,j}$ according to Algorithm~\ref{algo:pBcurve}. At this point, it is not evident how to determine the parameters of the vector basis functions, i.e. $m_1$, $m_2$, $q_1$, $q_2$, $u_{1,i}$, and $u_{2,j}$. This is because we need to take the assumptions of \autoref{eq:u=t}--\autoref{eq:pbasis_knot} into account.

We address this issue by \autoref{algo:knots}. We apply this algorithm to the coordinates $t_1$ and $t_2$ successively. The resulting knots of the basis function can actually satisfy \autoref{eq:u=t}--\autoref{eq:pbasis_knot}.

\begin{algorithm}[H]
  \caption{Determination of the knots of the vector basis function.}
  \label{algo:knots}
  \begin{algorithmic}[1]
    \STATE Let $p$, $n$, and $T:=\{t_0,\ldots,t_{n+p}\}$ be the predefined parameters of the periodic B-spline surface for either $t_1$ or $t_2$ coordinate. Then, once the degree $q$ is given, this algorithm determines the parameter $m$ and the knot vector $U:=\{u_0,\ldots,u_{m+q}\}$ of a vector basis function (i.e. $\Mp$ in \autoref{eq:pMdiv} or $\Np$ in \autoref{eq:pNdiv}) for the underlying coordinate.
    \IF {$p \ge q$}
    
    \STATE We need to use less number of knots for $U$ than that for $T$. To this end, we may remove $p-q$ ($=:r$) knots from both ends of $T$ to yield
    \begin{eqnarray*}
      U := \{t_{r},\ldots,t_{r+q-1}\ldots,t_{n+p-r} \}=:\{u_0,\ldots,u_{m+q}\},
    \end{eqnarray*}
    where $m$ is determined as $n-r$.
    \ELSIF {$p < q$}
    
    \STATE We need to use more number of knots for $U$ than that for $T$. To this end, we may add $q-p$ ($=:s$) knots to the both ends of $T$ to yield
    \begin{eqnarray*}
      U = \{\tau^-_0, \ldots, \tau^-_{s-1}, t_0, \ldots, t_{n+p}, \tau^+_{0}, \ldots, \tau^+_{s-1} \} =:\{u_0,\ldots,u_{m+q}\},
    \end{eqnarray*}
    where $m$ is determined as $n+s$ and the added knots $\tau^-_i$ ($<0$) and $\tau^+_i$ ($>1$) are defined as
    \begin{eqnarray*}
      \tau^-_i := t_{n-p-i} - (t_{n-p} - t_0),\quad
      \tau^+_i := t_{2p + i} + (t_{n+p} - t_{2p})\quad(i=1,\ldots,s).
    \end{eqnarray*}
    
    \ENDIF
    
    \STATE The resulting knots actually satisfy $u_q=t_p$, $u_m=t_n$ and $\Delta u_i=\Delta u_{i+m-q}$ ($i=0,\ldots,2q-1$), which correspond to \autoref{eq:u=t} and \autoref{eq:pbasis_knot}, respectively.
    
    \IF {$m<2q$}
    \STATE We repeat applying a uniform knot insertion to $U$ so that $m\ge 2q$ in \autoref{eq:alotofknots} is satisfied.
    \ENDIF
    
  \end{algorithmic}
\end{algorithm}

To explain \autoref{algo:knots}, \autoref{fig:pNdiv_knot}(a) shows an example of $T$ and the associated B-spline functions in the case of $p=3$ and $n=10$. \autoref{fig:pNdiv_knot}(b) and (c) show $U$ and the associated B-spline functions in the case of $q=2$ and $4$, respectively.

\begin{figure}[H]
  \centering
  \includegraphics[width=.6\textwidth]{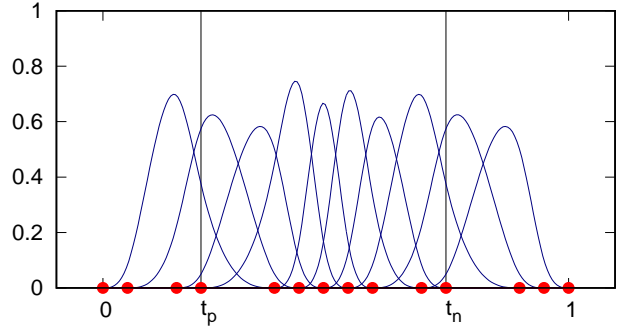}\\
  (a)~Knots of $T$ and the associated B-spline functions in the case of $p=3$, $n=10$, and $T=\{0, \frac{1}{19}, \frac{3}{19}, \frac{4}{19}, \frac{7}{19}, \frac{8}{19}, \frac{9}{19}, \frac{10}{19}, \frac{11}{19}, \frac{13}{19}, \frac{14}{19}, \frac{17}{19}, \frac{18}{19}, 1\}$.
  \includegraphics[width=.6\textwidth]{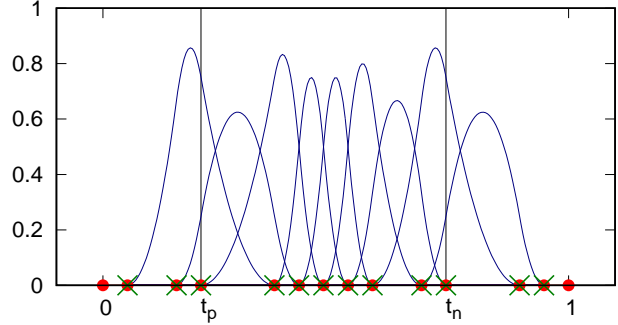}\\
  (b)~Knots of $U$ and the associated B-spline functions in the case of $q=2$. Here, \autoref{algo:knots} gives $m=n-(p-q)=9$ and $U=\{\frac{1}{19}, \frac{3}{19}, \frac{4}{19}, \frac{7}{19}, \frac{8}{19}, \frac{9}{19}, \frac{10}{19}, \frac{11}{19}, \frac{13}{19}, \frac{14}{19}, \frac{17}{19}, \frac{18}{19}\}$.
  \includegraphics[width=.6\textwidth]{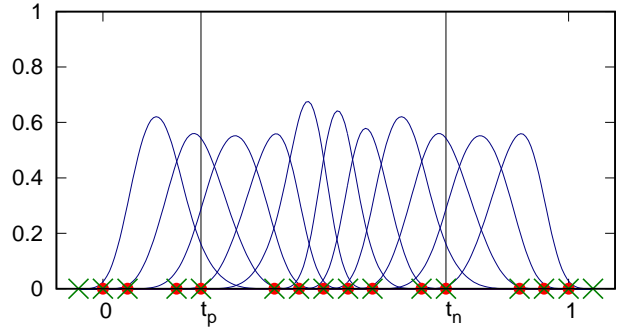}\\
  (c)~Knots of $U$ and the associated B-spline functions in the case of $q=4$. Here, \autoref{algo:knots} gives $m=n+(q-p)=11$ and $U=\{-\frac{1}{19}, 0, \frac{1}{19}, \frac{3}{19}, \frac{4}{19}, \frac{7}{19}, \frac{8}{19}, \frac{9}{19}, \frac{10}{19}, \frac{11}{19}, \frac{13}{19}, \frac{14}{19}, \frac{17}{19}, \frac{18}{19}, 1, \frac{20}{19}\}$.
  \caption{Example of constructing the knot vector $U$ from $T$. The red marks `$\bullet$' denote the knots in $T$, while the green marks `$\times$' denote the knots in $U$.}
  \label{fig:pNdiv_knot}
\end{figure}

\section{Galerkin IGBEM}\label{s:igbem}

We will establish the (Galerkin) IGBEM for the problem in \autoref{eq:QPBVP}, i.e. the 3D doubly-periodic problems in a layered structure, on the basis of (i)~the periodic B-spline surface (investigated in \autoref{s:surface}) and (ii)~the quasi-periodic vector basis and weight functions (\autoref{s:basis}). First, we will outline the discretisation of the variational equations in \autoref{eq:mom}. Successively, the evaluation of the boundary integrals in the resulting coefficient matrix will be mentioned.

\subsection{Discretisation of BIE}\label{s:disc}

First, we express an interface $S_i$ ($i=0,\ldots,\nd-2$) with a periodic B-spline surface according to \autoref{s:surf_pBsurface}.

Second, we discretise the electric and magnetic current densities as in \autoref{eq:approx_current}. To this end, we define $\J^{(i)}$ (respectively, $\M^{(i)}$) as $\J_i$ (respectively, $\M_i$) on $S_i$; then, $\J_{i+1}\equiv-\J^{(i)}$ (respectively, $\M_{i+1}\equiv-\M^{(i)}$) holds from the boundary condition in \autoref{eq:qpbBC1} (respectively, \autoref{eq:qpbBC2}). Then, $\J^{(i)}$ and $\M^{(i)}$ are represented as
\begin{eqnarray}
  \J^{(i)}(\vx) \simeq \sum_{a=1}^{n_a^{(i)}} J_a^{(i)} \Np_a(\vx),\quad
  \M^{(i)}(\vx) \simeq \sum_{a=1}^{n_a^{(i)}} M_a^{(i)} \Np_a(\vx),
  \label{eq:approx_JM}
\end{eqnarray}
where $J_a^{(i)}$ and $M_a^{(i)}$ are $2n_a^{(i)}$ unknown coefficients on $S_i$.

Third, we choose $\Np_a$ as the weight function $\bm{w}$ on every interface.

Successively, we substitute $\J^{(i)}$, $\M^{(i)}$, and $\bm{w}=\Np_a$ into the variational integral equations in \autoref{eq:mom}. As a result, we can obtain a set of $N$ ($:=\sum_{i=0}^{\nd-2}2n_a^{(i)}$) linear equations. We solve it by the LU decomposition. Since the equations in \autoref{eq:mom} are related to the three interfaces, i.e. $S_{i-1}$, $S_i$ and $S_{i+1}$, the coefficient matrix has a certain block structure. However, we will not utilise the structure in the solution.

\subsection{Evaluation of surface integrals}\label{s:integral}

We describe the way to evaluate the double integrals in the operators $\mathscr{L}^\p_d$ in \eqref{eq:opL} and $\mathscr{K}^\p_d$ in \eqref{eq:opK}. The singularity of these integrals are $O(r^{-1})$, where $r:=|\vx-\vy|$.\footnote{This is evident for $\mathscr{L}^\p_d$ because of the regularisation (recall \autoref{s:regularisation}). Meanwhile, the singularity of $\mathscr{K}^\p_d$ seems $O(r^{-2})$ at first. However, since the weight function $\bm{w}_a$ is perpendicular to $\vt\times\nabla_y \Gp_d$, the singularity is actually $O(r^{-1})$. This can be confirmed by expanding the integrand $\bm{w}_a\cdot\vt\times\nabla_y \Gp_d$ in terms of $\vy$ around the vicinity of $\vx$.} Therefore, on an interface $S$, we may consider the following type of double-surface integral:
\begin{eqnarray*}
  \int_S {f}(\vx) \int_S K^\p(\vx-\vy) g(\vy) \diff S_y \diff S_x
  = \sum_i \sum_j \underbrace{\int_{\Sb_i} f(\vx) \int_{\Sb_j} K^\p(\vx-\vy) g(\vy) \diff S_y \diff S_x}_{\displaystyle I_{ij}},
\end{eqnarray*}
where $\Sb_i$ stands for the $i$-th B\'ezier element (recall Remark~\ref{theo:pBsurfprop} in \autoref{s:surf_pBsurface}) and $f$ and $g$ are regular functions. Also, $K^\p$ has the singularity of $O(r^{-1})$ and is quasi-periodic. In what follows, we describe the case of $K^\p = \Gp_d$. In what follows, the subscript $d$ will be dropped for simplicity.

The way of computing $I_{ij}$ can be classified to the following three cases:
\begin{enumerate}
  
\item Singular case I
  
  The underlying integral $I_{ij}$ is singular if $\Sb_i$ and $\Sb_j$ are identical or share an edge or vertex. In this case,  we split $I_{ij}$ into the singular and regular parts as follows:
  \begin{eqnarray*}
    I_{ij} = I^{\rm sing}_{ij} + I^{\rm reg}_{ij},
  \end{eqnarray*}
  where
  \begin{eqnarray*}
    &&I^{\rm sing}_{ij} := \int_{\Sb_i} f(\vx) \int_{\Sb_j} G(\vx-\vy) g(\vy) \diff S_y \diff S_x,\\
    &&I^{\rm reg}_{ij} := \int_{\Sb_i} f(\vx) \int_{\Sb_j} \sum_{\bm{\mu}\in\mathbb{Z}^2\setminus(0,0)}\e^{\imath\vk^\inc\cdot\bm{p}^{(\bm{\mu})}}G(\vx-(\vy + \bm{p}^{(\bm{\mu})})) g(\vy) \diff S_y \diff S_x .
  \end{eqnarray*}
  Then, we compute $I^{\rm sing}_{ij}$ by the Frangi's method~\cite{frangi2002}, which can evaluate the double-surface integral in the four dimensional space simultaneously with the help of the Duffy's variable transformation. On the other hand, we apply the Gauss-Legendre (GL) formula to the regular integral $I^{\rm reg}_{ij}$. Basically, we evaluate $I^{\rm reg}_{ij}$ as
  \begin{eqnarray*}
    I^{\rm reg}_{ij} = \int_{\Sb_i} f(\vx) \int_{\Sb_j} \left( \Gp(\vx-\vy) - G(\vx-\vy)\right) g(\vy) \diff S_y \diff S_x,
  \end{eqnarray*}
  where $\Gp$ is computed by the Ewald's method (\ref{s:ewald}). It should be noted that we cannot compute $\Gp$ in the case of $\vx=\vy$. For this case, we use the explicit expression of the singular part of $G^\p$ and evaluate it by the Ewald's method~\cite[Theorem 3.8]{arens2010}.
  
\item Singular case II

  Even if $\Sb_i$ is separated from $\Sb_j$ in the primary cell, $I_{ij}$ is singular if $\Sb_i$ shares an edge or vertex with a replica of $\Sb_j$. \autoref{fig:singperiodic} illustrates three examples. In the LHS case, the $\bm{p}^{(1,0)}$-component of $G^\p$, i.e. $\frac{\mathrm{e}^{\imath k_i |\vx - (\vy + L_1\bm{e}_1)|}}{4\pi|\vx - (\vy + L_1\bm{e}_1)|}$ diverges when $\vx = \vy + L_1\bm{e}_1$. Therefore, similarly to the previous case, we apply the add-and-subtract technique together with the Frangi's method to the present singular integral $I_{ij}$.

\item Non-singular case

  If $\Sb_i$ is separated from both $\Sb_j$ and $S^\p$, $I_{ij}$ is non-singular. Then, we directly exploit the GL formula to evaluate $I_{ij}$.

\end{enumerate}

\begin{figure}[H]
  \centering
  \begin{tabular}{ccc}
    \iffalse    
    \includegraphics[height=.18\textheight]{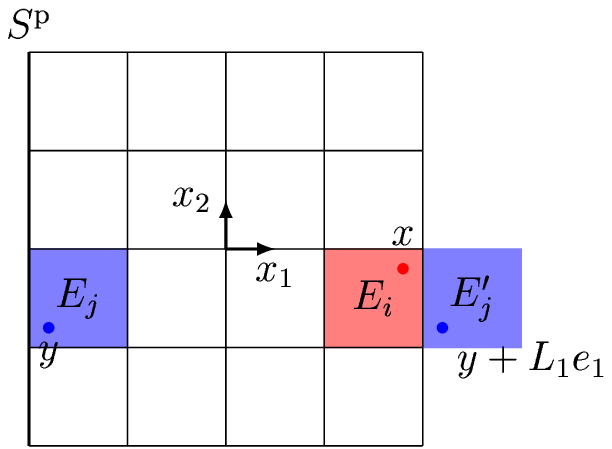}
    &\includegraphics[height=.18\textheight]{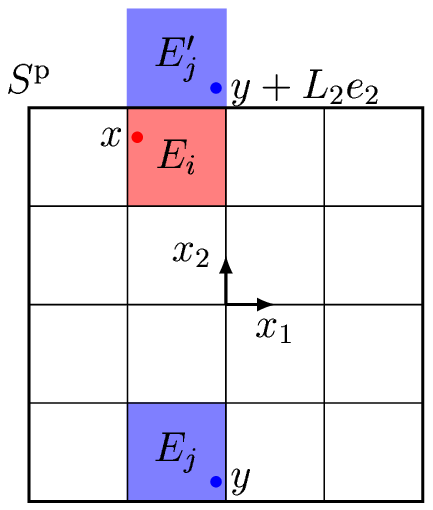}
    &\includegraphics[height=.18\textheight]{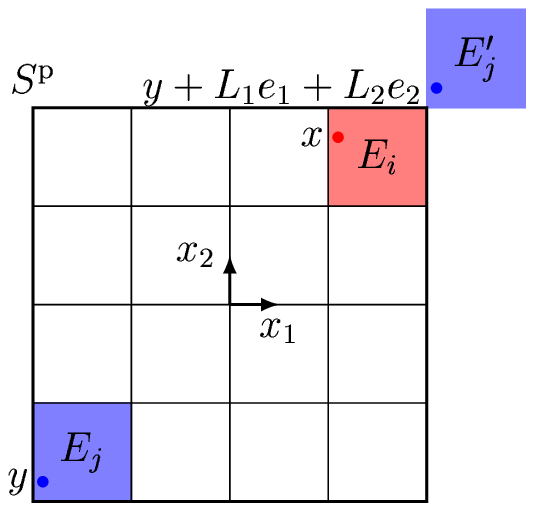}
    \else

\begin{tikzpicture}[scale=0.65]

  \draw [thick,white] (-0.1, -0.1) rectangle (5.5, 5.5); 
 
  \draw [fill, blue!40!white] (0.0, 1.0) rectangle (1.0, 2.0);
  \draw [] (0.5, 1.5) node {$E_j$};
  \draw [fill, blue] (0.2, 1.2) circle (0.05);
  \draw [below] (0.2, 1.1) node {$\bm{y}$};

  \draw [fill, red!40!white] (3.0, 1.0) rectangle (4.0, 2.0);
  \draw [] (3.5, 1.5) node {$E_i$};
  \draw [fill, red] (3.8, 1.8) circle (0.05);
  \draw [above] (3.8, 1.9) node {$\bm{x}$};

  \draw [fill, blue!40!white] (4.0, 1.0) rectangle (5.0, 2.0);
  \draw [] (4.5, 1.5) node {$E_j'$};
  \draw [fill, blue] (4.2, 1.2) circle (0.05);
  \draw [below right] (4.1, 1.1) node {$\bm{y}+L_1\bm{e}_1$};

  \draw [-latex, thick] (2.0, 2.0)--(2.5, 2.0);
  \draw [below] (2.5, 2.0) node {$x_1$};
  \draw [-latex, thick] (2.0, 2.0)--(2.0, 2.5);
  \draw [left] (2.0, 2.5) node {$x_2$};

  \draw [thick] (0.0, 0.0) rectangle (0.0, 4.0);
  \draw [above] (0.0, 4.0) node {$S^{\rm p}$};

  \draw [step=1.0] (0.0, 0.0) grid (4.0, 4.0);

\end{tikzpicture}
    &

\begin{tikzpicture}[scale=0.65]

  \draw [thick,white] (-0.1, -0.1) rectangle (5.5, 5.5); 

  \draw [fill, blue!40!white] (1.0, 0.0) rectangle (2.0, 1.0);
  \draw [] (1.5, 0.5) node {$E_j$};
  \draw [fill, blue] (1.9, 0.2) circle (0.05);
  \draw [right] (1.9, 0.3) node {$\bm{y}$};

  \draw [fill, red!40!white] (1.0, 3.0) rectangle (2.0, 4.0);
  \draw [] (1.5, 3.5) node {$E_i$};
  \draw [fill, red] (1.1, 3.7) circle (0.05);
  \draw [left] (1.1, 3.7) node {$\bm{x}$};

  \draw [fill, blue!40!white] (1.0, 4.0) rectangle (2.0, 5.0);
  \draw [] (1.5, 4.5) node {$E_j'$};
  \draw [fill, blue] (1.9, 4.2) circle (0.05);
  \draw [right] (1.9, 4.4) node {$\bm{y}+L_2\bm{e}_2$};

  \draw [-latex, thick] (2.0, 2.0)--(2.5, 2.0);
  \draw [below] (2.5, 2.0) node {$x_1$};
  \draw [-latex, thick] (2.0, 2.0)--(2.0, 2.5);
  \draw [left] (2.0, 2.5) node {$x_2$};

  \draw [thick] (0.0, 0.0) rectangle (4.0, 4.0);
  \draw [above] (0.0, 4.0) node {$S^{\rm p}$};

  \draw [step=1.0] (0.0, 0.0) grid (4.0, 4.0);

\end{tikzpicture}
    &

\begin{tikzpicture}[scale=0.65]

  \draw [thick,white] (-0.1, -0.1) rectangle (5.5, 5.5); 

  \draw [fill, blue!40!white] (0.0, 0.0) rectangle (1.0, 1.0);
  \draw [] (0.5, 0.5) node {$E_j$};
  \draw [fill, blue] (0.1, 0.2) circle (0.05);
  \draw [left] (0.1, 0.2) node {$\bm{y}$};

  \draw [fill, red!40!white] (3.0, 3.0) rectangle (4.0, 4.0);
  \draw [] (3.5, 3.5) node {$E_i$};
  \draw [fill, red] (3.1, 3.7) circle (0.05);
  \draw [left] (3.1, 3.7) node {$\bm{x}$};

  \draw [fill, blue!40!white] (4.0, 4.0) rectangle (5.0, 5.0);
  \draw [] (4.5, 4.5) node {$E_j'$};
  \draw [fill, blue] (4.1, 4.2) circle (0.05);
  \draw [left] (4.1, 4.3) node {$\bm{y}+L_1\bm{e}_1+L_2\bm{e}_2$};

  \draw [-latex, thick] (2.0, 2.0)--(2.5, 2.0);
  \draw [below] (2.5, 2.0) node {$x_1$};
  \draw [-latex, thick] (2.0, 2.0)--(2.0, 2.5);
  \draw [left] (2.0, 2.5) node {$x_2$};

  \draw [thick] (0.0, 0.0) rectangle (4.0, 4.0);
  \draw [above] (0.0, 4.0) node {$S^{\rm p}$};

  \draw [step=1.0] (0.0, 0.0) grid (4.0, 4.0);

\end{tikzpicture}
    \fi
  \end{tabular}
  \caption{Examples of the singular case II. In each sub-figure, $4\times 4$ squares represent B\'ezier elements on the underlying interface $S$, which is in the primary cell. A B\'ezier element $\Sb_i$, which is coloured in red, shares an edge or vertex with a replica $\Sb_j'$ of another element $\Sb_j$ through the periodic boundary $S^\p$.}
  \label{fig:singperiodic}
\end{figure}

\section{Numerical experiments}\label{s:num} 

We assess the proposed IGBEM numerically through two examples.

\subsection{Verification --- Problem 1}\label{s:num_ex1}

To verify the accuracy of the developed IGBEM, we solved a scattering problem due to plain parallel dielectric substrates consisting of five layers, where the material constants were virtually given as $\varepsilon_0=1$, $\varepsilon_1=2.25$, $\varepsilon_2=4$, $\varepsilon_3=2.25$, $\varepsilon_4=1$ [\si{\farad/\meter}], and $\mu_0=\cdots=\mu_4=1$ [\si{\henry/\meter}] (\autoref{fig:numex_layer}). We let $L_1=L_2=1$ [\si{\meter}], although these periods are arbitrary in this configuration. Letting the angular frequency $\omega$ be $8$ [\si{\radian/\meter}], we considered an oblique incident planewave of spherical angles $\theta=\phi=\frac{\pi}{4}$ \si{\radian}, i.e.
\begin{eqnarray*}
  &&\E^\inc(\vx) = \frac{1}{\sqrt{3}}\left(1,1,1\right)^{\rm T}~[\si{\volt/\meter}],\quad
  \H^\inc(\vx) = \frac{1}{\sqrt{2}}\left(1,-1,0\right)^{\rm T}~[\si{\ampere/\meter}],\\
  &&\bm{k}^\inc(\vx) = \frac{k_0}{\sqrt{3}}\left(1,1,-1\right)^{\rm T}~[\si{\radian/\meter}].
\end{eqnarray*}
A similar multi-layer problem was solved by Otani et al. in the case of the RWG basis function~\cite{otani2008}.

\begin{figure}[H]
  \centering
  \includegraphics[width=.35\textwidth]{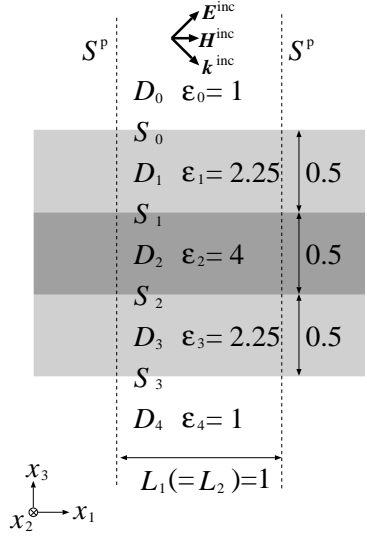}
  \caption{Scattering problem due to the parallel dielectric layers (\textbf{Problem 1}). The magnetic permeability is assumed to be one in all the layers. See \autoref{s:num_ex1}.}
  \label{fig:numex_layer}
\end{figure}

To perform the IGBEM, we generated each interface with a periodic B-spline surface using $p_1=p_2=1$ ($=:p$) and $n_1=n_2=6$ ($=:n$). We call the generated surface the initial mesh or Mesh0. By inserting knots into Mesh0 uniformly, we obtained a finer mesh or Mesh1. In the same way, we generated Mesh2 from Mesh1. \autoref{fig:numex_layer_mesh} shows the knot lines of every mesh. In this figure, a square consisting of four adjacent knot lines represents a B{\'e}zier element.

\begin{figure}[H]
  \centering
  \begin{tabular}{ccc}
    \includegraphics[width=.3\textwidth]{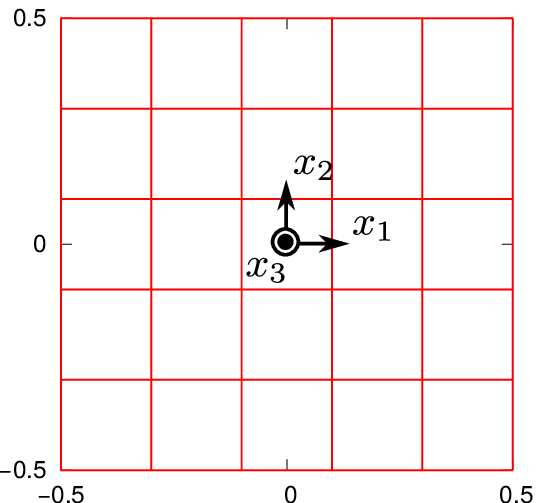}
    &\includegraphics[width=.3\textwidth]{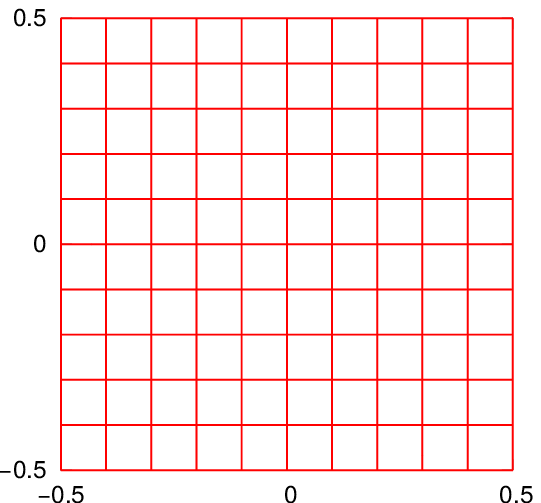}
    &\includegraphics[width=.3\textwidth]{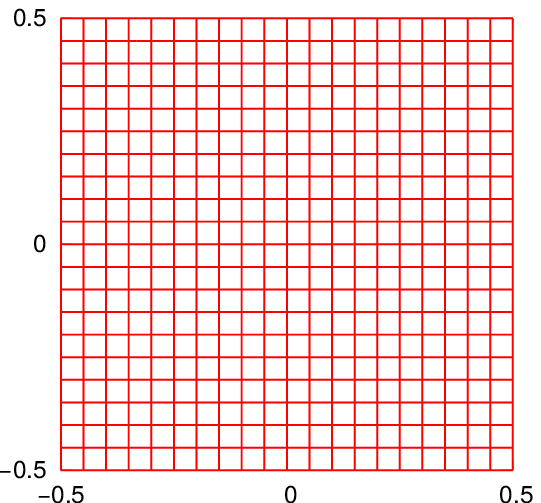}\\
    Mesh0 & Mesh1 & Mesh2
  \end{tabular}  
  \caption{Meshes (knot lines) used for \textbf{Problem 1}. See \autoref{s:num_ex1}.}
  \label{fig:numex_layer_mesh}
\end{figure}

Regarding the vector basis function $\Np$ in \autoref{eq:pNdiv}, we varied the degree $q$ ($:=q_1=q_2$) from 1 to 4. Here, the knot vector ($U$) was determined according to Algorithm~\ref{algo:knots} for every $q$.

Regarding the GL quadrature used for every integral variable, we used the 12 points formula in the case of (nearly-)singular integrals, while we did the 4 points formula in the case of non-singular integrals. We note that, as usual in the conventional BEM, it would be possible to determine the number of quadrature points adaptively. In addition, a technique of sub-division is helpful to improve the accuracy of a surface integral~\cite{taus2015,takahashi2022shape}. However, these are not considered in this study because they would need complicated implementations and could result in a high computational cost.

Regarding the Ewald's method, we set $10^{-14}$ to the threshold $\epsilon_{\rm ewald}$ in both \autoref{eq:ewald_E} and \autoref{eq:ewald_G}.

As a reference, we solved this problem with the transfer matrix (T-matrix) method~\cite{waterman1965}. Then, we measured the relative error $E_{\mathrm{rel}}$ of a surface current density, say $\vt$, computed by the IGBEM, denoted by $\vt_{\rm target}$, from one computed by the T-matrix method, denoted by $\vt_{\rm reference}$, with the $L^2$ norm, i.e.
\begin{eqnarray*}
  E_{\mathrm{rel}} := \frac{\sqrt{\sum_{i=0}^3\int_{S_i} |\vt_{\rm target} - \vt_{\rm reference}|^2 \dS}}{\sqrt{\sum_{i=0}^3\int_{S_i} |\vt_{\rm reference}|^2 \dS}},
\end{eqnarray*}
where $\vt$ is $\J^{(i)}$ or $\M^{(i)}$; recall \autoref{eq:approx_JM}. Here, the surface (non-singular) integrals in the RHS were evaluated by using the tensor product of the 4 points GL quadrature rules.

\autoref{fig:pfilm_error} plots $E_{\mathrm{rel}}$ for $\J$ and $\M$ against the inverse of the representative length $h$ of B{\'e}zier elements, i.e. $h:=\sqrt{\max_i(\text{Area of $\Sb_i$})}$. We can observe that the relative error becomes smaller as the mesh becomes finer for every degree $q$. In addition, the asymptotic convergence rate was nearly $O(h^{-q})$. These tendencies are consistent to the non-periodic case~\cite{simpson2018,dolz2018}.  The present result verifies the accuracy of the proposed IGBEM.

\begin{figure}[h]
  \centering
  \begin{tabular}{cc}
    \includegraphics[width=.45\textwidth]{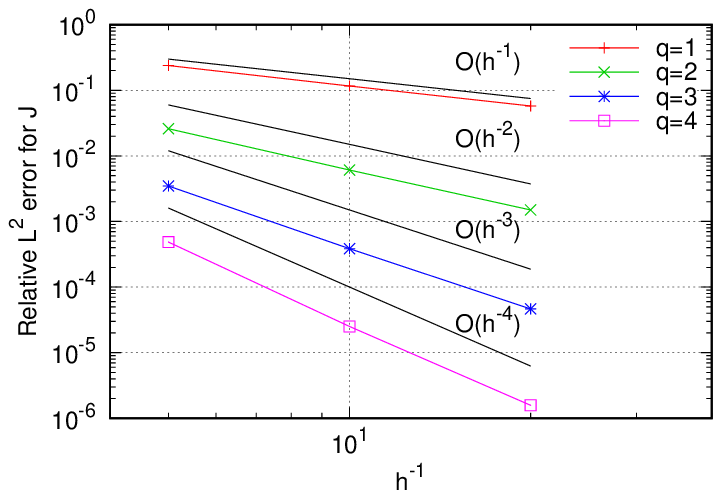}
    & \includegraphics[width=.45\textwidth]{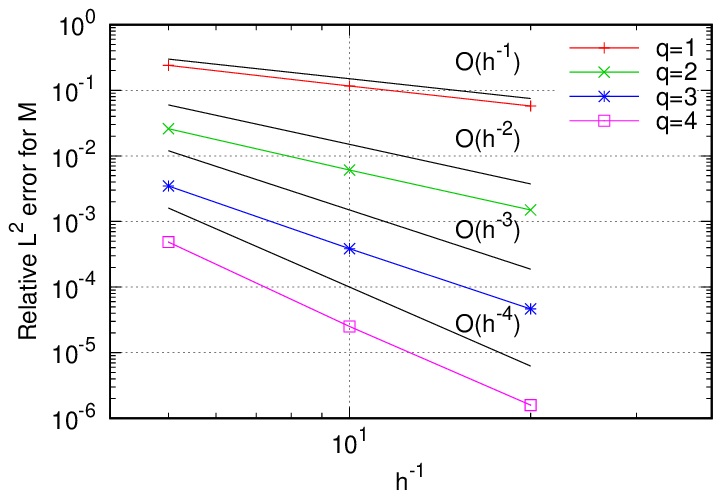}
  \end{tabular}
  \caption{Relative $L^2$ error $E_{\mathrm{rel}}$ for the surface current densities $\J$ (left) and $\M$ (right) in \textbf{Problem 1}. The black lines show the asymptotic rate of convergence. See \autoref{s:num_ex1}.}
  \label{fig:pfilm_error}
\end{figure}

In Problem 1, the representative length is regarded as $L_1=L_2=1$ [\si{\meter}], whereas the wavelength in $D_0$ is $\frac{2\pi}{k_0}=0.785$ [\si{\meter}]. Hence, the present problem can be considered as a low-frequnecy one. However, when performing the present IGBEM, there is no restriction on the frequency (wavenumber) from the theoretical viewpoint. We can expect a certain accuracy according to the number of (B{\'e}zier) elements per wavelength, as in the ordinary BEM. Nevertheless, it would be difficult to apply the present IGBEM to high-frequency problems owing to possible large computation time and memory usage. This issue could be resolved by, for example, enhancing the pFMM~\cite{otani2008} from the RWG basis function to the present B-spline basis functions $\Np$ and $\Mp$, although the enhancement is left as a future work.

As observed in \autoref{fig:pfilm_error}, there was no significant difference between the numerical accuracy of $\J$ and $\M$ in the following problem. We will thus show only the result of $\J$ hereafter.

\subsection{Discussions}

\subsubsection{Applicability to a non-plane surface --- Problem 2}\label{s:num_ex2}

In order to see that the present IGBEM can work for non-plane surfaces, we considered a sinusoidal surface between two layers. Letting $L_1=L_2=1$, $n_1=n_2=9$ ($:=n$) and $p_1=p_2=4$ ($:=p$), we gave the third component $(\bm{p}_{i,j})_3$ of the $(i,j)$-th control point $\bm{p}_{i,j}$ as
\begin{eqnarray*}
  (\bm{p}_{i,j})_3=0.3\cos\left(2\pi L_1 (\bm{p}_{i,j})_1\right) \cos\left(2\pi L_2 (\bm{p}_{i,j})_2\right)\quad(0\le i<n_1-p_1,\ 0\le j<n_2-p_2).
\end{eqnarray*}
Meanwhile, the horizontal components of $\bm{p}_{i,j}$ were determined according to \autoref{eq:pBcurve_cond3}. The generated periodic B-spline surface is called Mesh0, which is shown in \autoref{fig:num_ex2_mesh0}. Note that the control points are generally apart from the generated B-spline surface due to the nature of the B-spline functions.

Similarly to Problem 1 in \autoref{s:num_ex1}, we generated a sequence of meshes, i.e. Mesh0--3, by the knot insertion.

\begin{figure}[H]
  \centering
  \includegraphics[width=.5\textwidth]{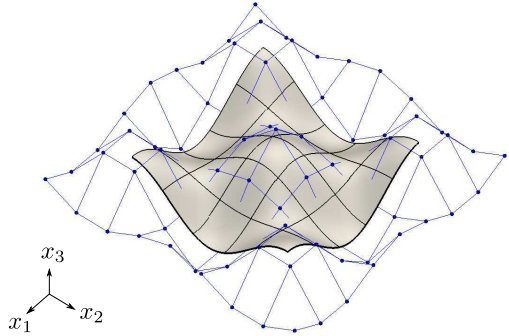}
  \caption{The initial mesh (Mesh0) used for \textbf{Problem 2}. Here, the blue points and lines show the control points and knot lines, respectively. See \autoref{s:num_ex2}.}
  \label{fig:num_ex2_mesh0}
\end{figure}

We used the same incident wave as the previous problem but the angular frequency $\omega$ was chosen as 10 [\si{\radian/\second}]. To compute the exact solution, we let both media be the same, i.e. $\varepsilon_1=\varepsilon_2=1$ and $\mu_1=\mu_2=1$. Then, the exact solutions of $\J$ and $\M$ are given as $\J^{\rm exact} = \vn\times\Hinc$ and $\M^{\rm exact}=\Einc\times\vn$, respectively, regardless the shape of the surface.

\autoref{fig:pcoscos_error} shows the relative $L^2$-error of $\J$ for $q$ ($:=q_1=q_2$) of 1 to 4. Similarly to the previous problem, the error converged monotonically as the mesh size $h$ decreased for every $q$. The asymptotic convergence rate resulted in $O(h^{-q})$ again.

\begin{figure}[H]
  \centering
  \includegraphics[width=.5\textwidth]{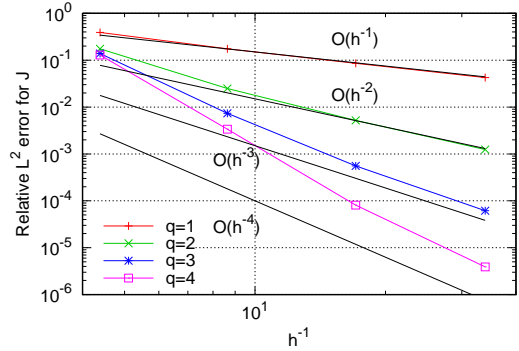}
  \caption{Relative $L^2$ error $E_{\mathrm{rel}}$ for the surface electric current density $\J$ in \textbf{Problem 2}. The result for the surface magnetic current density $\M$ was similar to $\J$. The black lines show the asymptotic rate of convergence. See \autoref{s:num_ex2}.}
  \label{fig:pcoscos_error}
\end{figure}

In the following subsections \ref{s:num_ex2_comptime}--\ref{s:num_ex2_p<q}, we will discuss the present IGBEM with respect to Problem 2 stated in \autoref{s:num_ex2}.

\subsubsection{Computation time}\label{s:num_ex2_comptime}

\autoref{fig:pcoscos_time} plots the computation time\footnote{The computation times were measured on a workstation with total 20 computing cores (CPU: Intel Xeon CPU E5-2687W v3; 3.10GHz). \label{footnote:resource}} against $N$, which slightly varies according to $q$ (recall the definition of $N$ in \autoref{s:disc})\footnote{In the case of the smallest $q$ of 1, $N$ takes the value of 100, 400, 1600 and 6400 for Mesh0, 1, 2 and 3, respectively. Meanwhile, in the case of the largest $q$ of 4, $N$ takes 160, 520, 1840 and 6880. \label{footnote:N}}. Since the computation of the periodic Green's function is the most time-consuming in our implementation, the computation time is proportional to $N^2$ rather than $N^3$, which is required for the LU decomposition.

In addition, a larger $q$ required a longer computation time. However, \autoref{fig:pcoscos_timeerr} indicates that a larger $q$ spent a shorter time for achieving a certain accuracy as far as the the error is sufficiently small, that is, $\lesssim 10^{-3}$. For example, to obtain the error of $10^{-4}$, $q=4$ requires about $10^4$ [\si{\second}]. On the other hand, it is estimated that $q=3$, $2$, and $1$ would need about $3\times 10^4$ [\si{\second}], $10^6$ [\si{\second}], and extremely much more time, respectively. This result is an achievement of the proposed IGBEM. It is significant that the proposed IGBEM provides a high-order approximation for 3D periodic problems. Meanwhile, it would be possible to enhance the conventional high-order RWG and RT basis functions from the non-periodic case to the periodic one according to, for example, Hu et al.~\cite{hu2011}. However, it makes the proposed IGBEM more advantageous than the conventional BEM that the IGBEM can handle a given geometry exactly.

\begin{figure}[H]
  \centering
  \includegraphics[width=.5\textwidth]{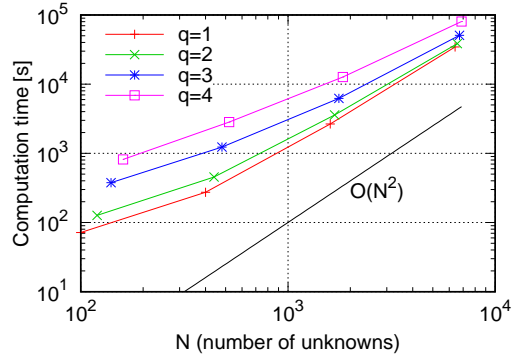}
  \caption{Computation time versus $N$ (i.e. the number of unknowns) in \textbf{Problem 2}. The black line shows the asymptotic rate of computation time. See \autoref{s:num_ex2_comptime}.}
  \label{fig:pcoscos_time}
\end{figure}

\begin{figure}[H]
  \centering
  \includegraphics[width=.5\textwidth]{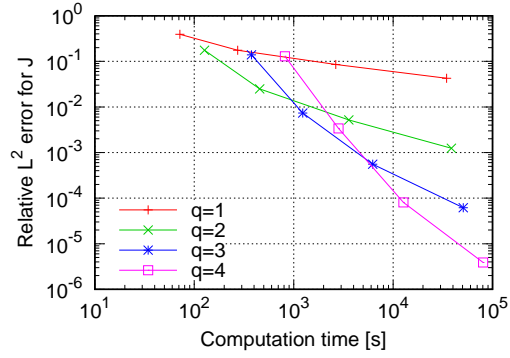}
  \caption{Computation time versus the relative $L^2$ error $E_{\mathrm{rel}}$ for the electric surface current density $\J$ in \textbf{Problem 2}. For each line, the four plots correspond four meshes, i.e. Mesh0 to Mesh3. See \autoref{s:num_ex2_comptime}.}
  \label{fig:pcoscos_timeerr}
\end{figure}

\subsubsection{Comparison of two vector basis functions}\label{s:num_ex2_N_vs_M}

So far, we have used the vector basis function $\Np$ in \autoref{eq:pNdiv} on a certain interface $S$. Here, $\Np$ is quasi-periodic for both the normal and tangential components to $\partial S$. Instead, we performed the IGBEM with $\Mp$ in \autoref{eq:pMdiv}, which is quasi-periodic only for the normal component.

\autoref{fig:NdivMdiverror} compares the relative $L^2$-error $E_{\mathrm{rel}}$ for the both vector basis functions. Clearly, the accuracy of $\Np$ is better than that of $\Mp$ for higher $q$s.

We note that, in the case of $q=1$, $\Mp$ is identical to $\Np$ whenever the surface is smooth everywhere. In addition, $\Mp$ of $q=1$ is equivalent to the rooftop function, i.e. the RT basis functions of the first (lowest) order.

\begin{figure}[H]
  \centering
  \includegraphics[width=.5\textwidth]{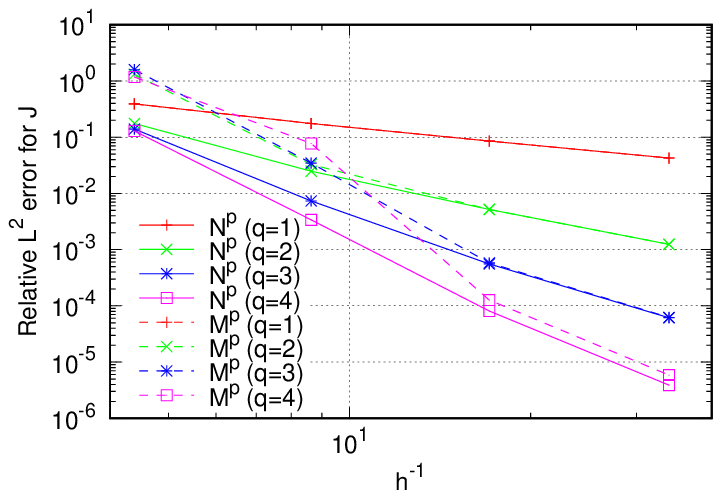}
  \caption{Relative $L^2$ error $E_{\mathrm{rel}}$ for the surface electric current density $\J$ for the vector basis function $\Np$ (solid lines) and $\Mp$ (dotted lines) in \textbf{Problem 2}. The result for the surface magnetic current density $\M$ was similar to that for $\J$. The black lines show the asymptotic rate of convergence. See \autoref{s:num_ex2_N_vs_M}.}
  \label{fig:NdivMdiverror}
\end{figure}

We further investigated the condition number of the system matrix in \autoref{fig:NdivMdivMatrix}.\footnote{We computed the condition number by utilising the LAPACK's routine ZLANGE after computing the inverse of the system matrix by ZGETRI.} The result shows that $\Np$ is more preferable than $\Mp$. As a matter of fact, when we actually solved the linear equations with the (non-restarted) GMRES~\cite{saad1986gmres}, $\Np$ converged faster than $\Mp$.

We can conclude that $\Np$ is superior than $\Mp$. Because any interface $S$ between two layers is smooth everywhere, which is assumed by \autoref{eq:pBcond1}, not only the normal component but also the tangential component (to $\partial S$) of the surface densities $\J$ and $\M$ must be continuous beyond the periodic boundary $S^\p$. Hence, $\Np$ is a more appropriate basis than $\Mp$ to approximate $\J$ and $\M$.

\begin{figure}[H]
  \centering
  \includegraphics[width=.5\textwidth]{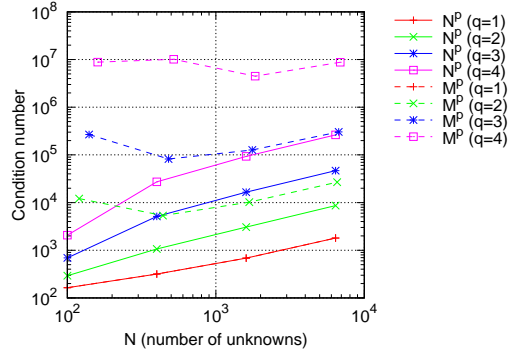}
  \caption{Condition numbers of the system matrix for the vector basis function $\Np$ (solid lines) and $\Mp$ (dotted lines) in \textbf{Problem 2}. See \autoref{s:num_ex2_N_vs_M}.}
  \label{fig:NdivMdivMatrix}
\end{figure}

\subsubsection{Inaccuracy in the case of $p<q$}\label{s:num_ex2_p<q}

In the non-periodic case, Buffa et al.~\cite{buffa2010} points out that the numerical accuracy of $p<q$ can be worse than that of $p\ge q$ theoretically and numerically, where $p$ ($:=p_1=p_2$) and $q$ ($:=q_1=q_2$) denote the (representative) degree of the surface and that of the vector basis function, respectively.\footnote{In addition, Buffa et al.~\cite{buffa2010} shows that the accuracy can be improved by increasing the value of $p$ to $q$ by the degree elevation.}

Following Buffa et al.~\cite{buffa2010}, we investigated the behaviour of our IGBEM when $p<q$. Specifically, we replaced $p=4$ with $2$ or $3$; correspondingly, we replaced $n=9$ with $7$ and $8$, respectively, to maintain the number of B\'ezier elements. The results of $p=2$ and $3$ are shown in \autoref{fig:pcoscos_error_p<q}. We can observe that the asymptotic convergence rate was worse than $O(h^{-q})$ when $p<q$. This result is consistent to Buffa et al.~\cite{buffa2010} for the non-periodic case.

\begin{figure}[H]
  \centering
  \begin{tabular}{cc}
    \includegraphics[width=.45\textwidth]{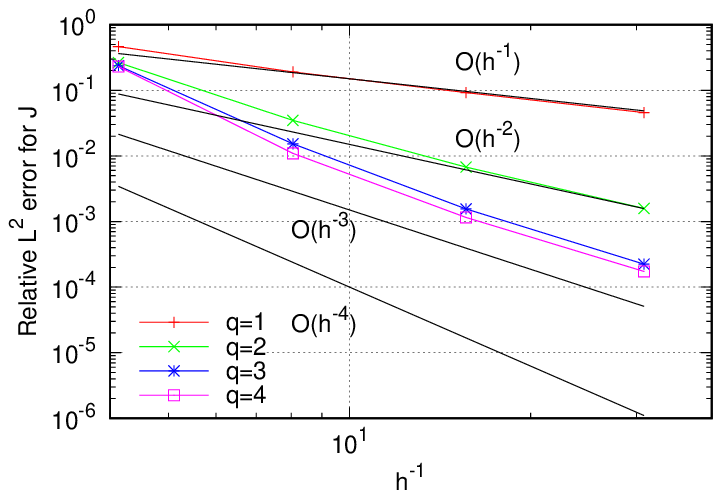}
    & \includegraphics[width=.45\textwidth]{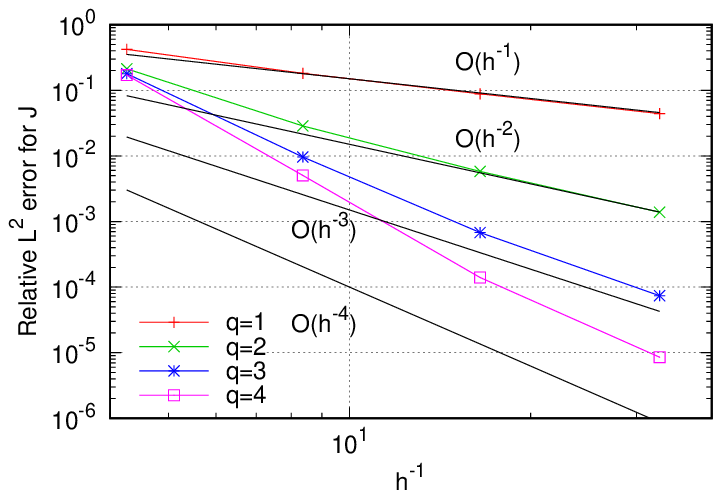}
  \end{tabular}
  \caption{Relative $L^2$ error $E_{\mathrm{rel}}$ for the surface electric current density $\J$ in the case of $p=2$ (left) and $p=3$ (right) in \textbf{Problem 2}. The results for the surface magnetic current density $\M$ was similar to that of $\J$. The black lines show the asymptotic rate of convergence. See \autoref{s:num_ex2_p<q}.}
  \label{fig:pcoscos_error_p<q}
\end{figure}

\section{Application to plasmonics} \label{s:plasmonics}

To demonstrate the applicability of our IGBEM, we simulated to excite propagating surface plasmons (SPs) on a doubly-periodic metal surface, that is, 2D diffraction grating.

\subsection{Problem setting}\label{s:plasmonics_problem}

Let us consider a layer structure consisting of air and silver. The interface between air and silver is doubly periodic. The unit structure is shown in \autoref{fig:grating_model}. Thus, the silver layer can be regarded as a 2D grating, although its thickness is infinitely large. Specifically, we construct the surface of the unit structure as the tensor product of a certain function $f(\cdot)$, i.e. $\frac{1}{H}f(x_1)f(x_2)$, where the magnitude is adjusted by a constant $H$. Here, the function (profile) $f$ is defined as a piecewise-linear function as shown in \autoref{fig:2Dgrating_model}, where the lengths $L$, $W$, $H$, and the angle $\alpha$ are the free parameters. They are given as $L=0.3$ (which corresponds to the periods $L_1$ and $L_2$), $W=0.1$, $H=0.05$ [\si{\micro\meter}], and $\alpha=70$ [\si{\degree}]. To generate the surface, we followed \autoref{s:surf_pBcurve}. Regarding the parameters of B-spline functions, we used $n_1=n_2=17$ and $p_1=p_2=2$.

In regard to the vector basis function $\Np$, we let the degrees $q_1$ and $q_2$ be 2. The knots are determined according to Algorithm \ref{algo:knots}.

\begin{figure}[H]  
  \centering
  \includegraphics[width=.5\textwidth]{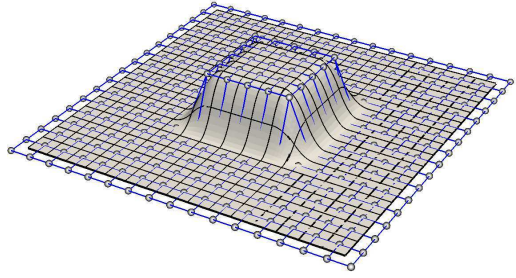}
  \caption{Periodic B-spline surface model for the present 2D grating (Subsection~\ref{s:plasmonics_problem}). Here, the grey points and the blue lines show the control points and the knot lines, respectively.}
  \label{fig:grating_model}
\end{figure}

\begin{figure}[H]
  \centering
  \includegraphics[width=.3\textwidth]{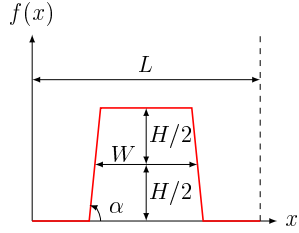}
  \caption{Profile function $f(x)$ and parameters of the present 2D grating.}
  \label{fig:2Dgrating_model}
\end{figure}

Let us consider the incident light of $\theta =\frac{\pi}{6}$ and $\phi = 0$ [\si{\radian}] in \autoref{eq:incident_EH}, where the amplitudes $\bm{a}^\inc$ and $\bm{b}^\inc$ are arbitrary in the following analysis as far as $|\bm{a}^\inc|/|\bm{b}^\inc|=\sqrt{\mu_0/\varepsilon_0}$ is fulfilled. Then, since $\bm{E}^\inc$ and $\bm{H}^\inc$ are parallel to the $x_1x_3$-plane and the $x_2$-axis, respectively, SPs can propagate in the $x_1$ direction. In this case, the dispersion relations of the SPs and that of the diffracted wave of the $m$-th mode (where $m\in\bbbz$) can be described in terms of the $x_1$-component of the wavevector, respectively, as follows~\cite{maier2007}:
\begin{eqnarray*}
  k_1^{\rm \pm SP}(\omega)=\pm\frac{\omega}{c_0}\sqrt{\frac{\varepsilon_{\rm Ag}(\omega)}{1 + \varepsilon_{\rm Ag}(\omega)}},\quad k^m_1(\omega) = \frac{\omega}{c_0}\sin\theta + \frac{2m\pi}{L},
\end{eqnarray*}
where, letting $\varepsilon_0=8.8541878\times 10^{-12}$ [\si{\farad/\meter}] and $\mu_0=1.2566370\times 10^{-6}$ [\si{\henry/\meter}] be the permittivity and the permeability of the vacuum (air), respectively, we denote the light speed in the air by $c_0:=(\varepsilon_0\mu_0)^{-1/2}$. Also, the relative permittivity of silver, denoted by $\varepsilon_{\rm Ag}(\omega)$, is a function of $\lambda_0$ ($\equiv\frac{2\pi c_0}{\omega}$; the wavelength in the vacuum), which is available from \cite{palik1985handbook}. \autoref{fig:dispertion2D} draws the dispersion relations, where only the diffracted light of $m=-1$, $0$, and $1$ are shown. We can see that $k_1^{\rm -SP}$ and $k_1^{-1}$ intersect at $\lambda_0\simeq 0.48$ [\si{\micro\meter}]. This indicates that the incident light of $\lambda_0\simeq 0.48$ [\si{\micro\meter}] can excite the SP of $m=-1$, which propagates in the $-x_1$ direction.

\begin{figure}[H]
  \centering
  \includegraphics[width=.6\textwidth]{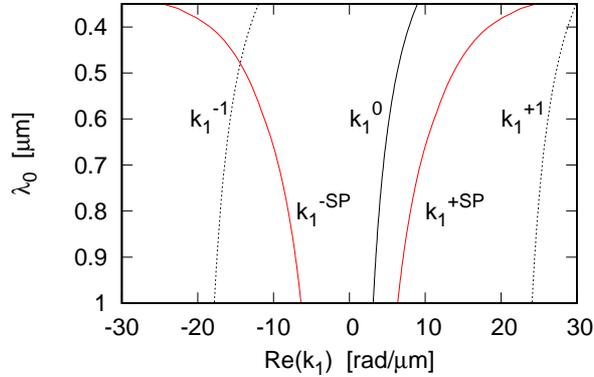}
  \caption{Dispersion relations of the surface plasmons ($k_1^{\pm {\rm SP}}$), the incident light ($k_1^0$), and the diffracted lights ($k_1^{-1,0,+1}$) for air/silver interface in the case of the incident light of $\theta=\frac{\pi}{6}$ and $\phi=0$ [\si{\radian}].}
  \label{fig:dispertion2D}
\end{figure}

\subsection{Numerical results}

We performed the IGBEM for a number of the incident wavelength $\lambda_0$ and computed the energy reflectance. \autoref{fig:R} shows that the energy reflectance has the minimal value at $\lambda_0=0.516$ [\si{\micro\meter}]. This value is close to the aforementioned wavelength where the dispersion curve of the diffracted light of $m=-1$ intersects with that of the backward-propagating plasmon.

In addition, we created a movie\footnote{This is available as the supplementary material of `\texttt{Hxz.gif}'.} of the magnitude of the real part of the dimensionless time-harmonic magnetic field, i.e. $|\bm{\mathcal{H}}(\vx,t)|$ where $\bm{\mathcal{H}}(\vx,t):=\mathrm{Re}\left[\frac{1}{|\bm{b}^\inc|}\bm{H}(\vx)\e^{-\imath\omega t}\right]$, for the $x_1x_3$-plane. In the movie, we can observe a surface wave propagating in the $-x_1$ direction. As a reference, we also created another movie\footnote{This is available as the supplementary material of `\texttt{Hyz.gif}'.} for the $x_2x_3$-plane. In this case, the field oscillates without any phase difference at all the points on the plane.

From these results, we can say that our simulation can excite the surface plasmon on the surface of the 2D grating successfully.

\begin{figure}[H]
  \centering
  \includegraphics[width=.6\textwidth]{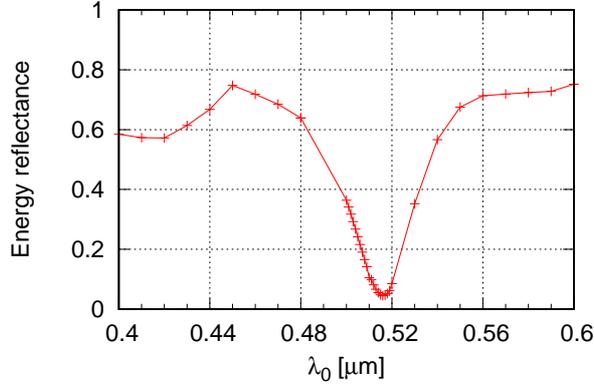}
  \caption{Energy reflectance.}
  \label{fig:R}
\end{figure}

\section{Conclusion}\label{s:conclusion}

We proposed an isogeometric boundary element method (IGBEM) for 3D doubly-periodic layered structures in electromagnetics. First, we expressed each interface between two layers with a rectangular B-spline surface with considering the double periodicity. To this end, we proposed an algorithm to generate a periodic B-spline curve as in Algorithm~\ref{algo:pBcurve}, which is a generalisation of the algorithm by Shimba et al.~\cite{shimba2015}. Second, following the pioneering work by Buffa et al.~\cite{buffa2010} as well as Simpson et al.~\cite{simpson2018} and D\"oltz et al.~\cite{dolz2018} for the non-periodic case, we constructed two types of the vector basis functions based on the B-spline functions (i.e. $\Np$ in \autoref{eq:pNdiv} and $\Mp$ in \autoref{eq:pMdiv}) for the present periodic problems. The construction is the central contribution of this paper. In the numerical analyses, we verified the accuracy of the implemented IGBEM, discussing some optional parameters and settings. Finally, we applied the IGBEM to a plasmonic simulation successfully.

We have some works to do in future. First, we need to accelerate our IGBEM. To this end, we are planning to employ the periodic fast multipole method (FMM) by Otani et al.~\cite{otani2008}, which considers the RWG basis function~\cite{rao1982}. As seen in the reference~\cite{takahashi2012}, if we adjusted the procedures of both creation of the moments and evaluation by the local coefficients for the current B-spline-based discretisation, it would be possible to apply the periodic FMM~\cite{otani2008} to the current IGBEM. However, this is out of the scope of this paper.

Second, we desire to perform the shape optimisation based on our (accelerated) IGBEM. In general, IGBEM is suitable for shape optimisation because changing the shape of a surface to be designed can be performed by modifying the locations of its control points. As a matter of fact, IGBEM has been exploited for shape optimisations (see a short survey in \cite{takahashi2019}). With a shape optimisation system, we are particularly interested in designing the surface-plasmon enhanced photovoltaic devices~\cite{atwater2010,polman2016} in order to realise a high efficient power generation in a very thin cell.

\section*{Acknowledgements}
This work was supported by JSPS KAKENHI Grant Number 18K11335. In addition, we would like to thank all the anonymous referees for their precious comments and suggestions.

\appendix

\section{Evaluation of the periodic Green's function}\label{s:ewald}

We summarise the Ewald's method to compute the periodic Green's function $\Gp_d$ in \autoref{eq:Gp}, where the wavenumber $k_d$ is assumed to be a real number. We will omit the domain index `$d$' from $\Gp_d$ and $k_d$ hereafter.

Following the reference~\cite{arens2010}, we split $\Gp$ int two parts as follows:
\begin{eqnarray*}
  \Gp(\vx-\vy)=G^{\mathrm{p1}}(\vx-\vy)+G^{\mathrm{p2}}(\vx-\vy),
\end{eqnarray*}
where
\begin{eqnarray}
  &&G^{\mathrm{p1}}(\vx-\vy):=\frac{a}{4\pi^{\frac{3}{2}}}\sum_{\bm{\nu}\in\mathbb{Z}^2}\e^{\imath\vk^\inc\cdot\bm{p}^{(\bm{\nu})}}\sum_{j=0}^\infty E_j^{(\bm{\nu})}(\vx-\vy),\label{eq:ewald_Gp1}\\
  &&G^{\mathrm{p2}}(\vx-\vy):=\frac{\imath}{4L_1L_2}\sum_{\bm{\nu}\in\mathbb{Z}^2} F^{(\bm{\nu})}(\vx-\vy).\label{eq:ewald_Gp2}
\end{eqnarray}%
Here, 
\begin{eqnarray*}
  E_j^{(\bm{\nu})}(\vx-\vy)&:=&\frac{1}{j!}\left(\frac{k}{2a}\right)^{2j} (a|\vx-(\vy + \bm{p}^{(\bm{\nu})}) |)^{2j-1}\Gamma\left(\frac{1}{2}-j,\ a^2|\vx - (\vy + \bm{p}^{(\bm{\nu})})|^2\right),\\
  F^{(\bm{\nu})}(\vx-\vy)&:=&\frac{\mathrm{e}^{\imath k \bm{d}^{(\bm{\nu})}\cdot(\vx-\vy)}}{k\rho^{(\bm{\nu})}} \left[\mathrm{e}^{-\imath k \rho^{(\bm{\nu})} (x_3 - y_3)}\mathrm{erfc}\left( - \frac{\imath k\rho^{(\bm{\nu})}}{2a} + (x_3 - y_3)a\right)\right.\\
    &&\left.+\mathrm{e}^{\imath k \rho^{(\bm{\nu})} (x_3 - y_3)}\mathrm{erfc}\left( - \frac{\imath k\rho^{(\bm{\nu})}}{2a} - (x_3 - y_3)a\right)\right],\\
  \bm{d}^{(\bm{\nu})} &:=& \frac{1}{k}\left(\frac{\beta_1 + 2\pi\nu_1}{L_1},\ \frac{\beta_2 + 2\pi\nu_2}{L_2},\ 0\right)^{\mathrm{T}},\quad
  \rho^{(\bm{\nu})} :=
     \begin{cases}
       \sqrt{1-|\bm{d}^{(\bm{\nu})}|^2}, &|\bm{d}^{(\bm{\nu})}|^2 < 1\\
       \imath\sqrt{|\bm{d}^{(\bm{\nu})}|^2 - 1}, &|\bm{d}^{(\bm{\nu})}|^2 > 1
     \end{cases},
\end{eqnarray*}
where $\Gamma$ and $\mathrm{erfc}$ denote the upper incomplete gamma function and the complementary error function, respectively. Also, the parameter $a$ can be chosen so that the computational cost can be minimised for a requested accuracy.

In the actual computation, we need to truncate the infinite serieses in \autoref{eq:ewald_Gp1} and \autoref{eq:ewald_Gp2}. To this end, we introduce the numbers $r_1$, $r_2$, and $n$ and, then, approximate $G^{\rm p1}$ and $ G^{\rm p2}$ as follows:
\begin{eqnarray*}
  &&G^{\rm p1}(\vx-\vy) \approx {G}_{r_1}^{\mathrm{p1}}(\vx-\vy) :=  \frac{a}{4\pi^{\frac{3}{2}}}\sum_{R=0}^{r_1}\ \sum_{\|\bm{\nu}\|_{\infty} = R} \e^{\imath \vk^\inc\cdot\bm{p}^{(\bm{\nu})}} H_n^{(\bm{\nu})}(\vx-\vy),\\
  &&G^{\rm p2}(\vx-\vy) \approx {G}_{r_2}^{\mathrm{p2}}(\vx-\vy) := \frac{\imath}{4L_1L_2}\sum_{R=0}^{r_2}\ \sum_{\|\bm{\nu}\|_{\infty} = R}F^{(\bm{\nu})}(\vx-\vy),
\end{eqnarray*}
where 
\begin{eqnarray*}
  H_n^{(\bm{\nu})}(\vx-\vy) := \sum_{j=0}^{n}E_j^{(\bm{\nu})}(\vx-\vy)
\end{eqnarray*}
and $\|\bm{\nu} \|_\infty := \max(|\nu_1|,|\nu_2|)$. Then, for a given set of $\vx$, $\vy$, and $\bm{\nu}$, we increase the value of $n$ until 
\begin{eqnarray}
  \frac{|{E}_{n+1}^{(\bm{\nu})} - {E}_n^{(\bm{\nu})}|}{|{E}_n^{(\bm{\nu})}|} < \epsilon_{\rm ewald}
  \label{eq:ewald_E}
\end{eqnarray}
is satisfied for a predefined tolerance $\epsilon_{\rm ewald}$. Successively, we increase the values of $r_1$ and $r_2$ until 
\begin{eqnarray}
  \frac{|{G}_{r_1+1}^{\mathrm{p1}} - {G}_{r_1}^{\mathrm{p1}}|}{|{G}_{r_1}^{\mathrm{p1}}|} < \epsilon_{\rm ewald}\quad\text{and}\quad
  \frac{|{G}_{r_2+1}^{\mathrm{p2}} - {G}_{r_2}^{\mathrm{p2}}|}{|{G}_{r_2}^{\mathrm{p2}}|} < \epsilon_{\rm ewald}
  \label{eq:ewald_G}
\end{eqnarray}
are satisfied, respectively.

The evaluation of $\nabla_y \Gp$, which appears in the operator $\mathscr{K}^\p_i$ in \autoref{eq:opK}, is similar to that of $\Gp$. 

\section{Proofs for Algorithm~\ref{algo:pBcurve}}\label{s:proof_surface}

This section mathematically justifies \autoref{algo:pBcurve} in \autoref{s:surf_pBcurve}. The algorithm is the direct consequence of Theorem~\ref{theo:pBcurve} below. This theorem will be proven through Lemmas~\ref{theo:translation}--\ref{theo:weak2} and Theorem~\ref{theo:pBcurve_weak}.

From the definition of the B-spline functions in \autoref{eq:CoxdeBoor}, we can show the following translation property:
\begin{lemma}{}
  Denote $t_{i+1}-t_i$ by $\Delta t_i$. If the knots near both ends are equidistant, i.e.
  \begin{eqnarray}
    \Delta t_i=\Delta t_{i+n-p}\quad(i=0,\ldots,2p-1),
    \label{eq:translation_assume}
  \end{eqnarray}
   then a B-spline curve presented by \autoref{eq:Bcurve} has the following translation property:
  \begin{eqnarray}
    \frac{\diff^k B^p_i}{\diff t^k}(t) \equiv \frac{\diff^k B^p_{n-p+i}}{\diff t^k}(t+T)\quad(i=0,\ldots,p-1,\ k=0,\ldots,p-1),
    \label{eq:translation}
  \end{eqnarray}
  where $T:=t_n-t_p$ denotes the length of the parameter domain $[t_p,t_n]$.
  \label{theo:translation}
\end{lemma}

\autoref{fig:pBsp} shows an example of the B-spline functions that satisfy \autoref{eq:translation_assume}, where $n=10$ and $p=3$. We can see that the shapes of $B_0^3$, $B_1^3$, and $B_{p-1}^3$ are the same as $B_{n-p}^3$, $B_{n-2}^3$, and $B_{n-1}^3$, respectively. Namely, \autoref{eq:translation} actually holds for $k=0$. Therefore, it is obvious that the first and higher derivatives also satisfy the translation property represented by \autoref{eq:translation}.

\begin{figure}[H]
  \centering
  \includegraphics[width=0.7\textwidth]{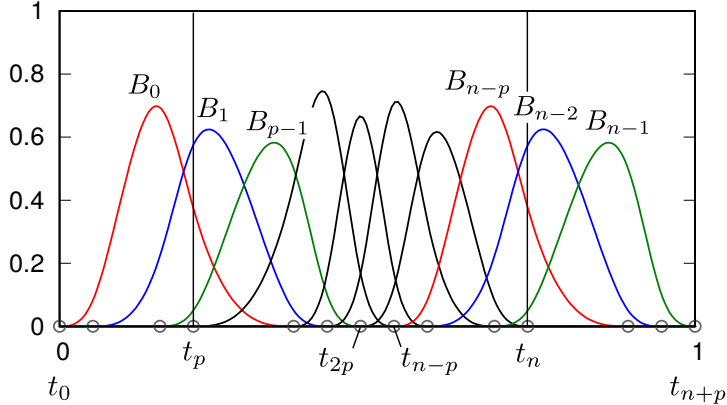}
  \caption{Example of B-spline functions of $p=3$ and $n=10$ with a knot vector $T:=\left\{0, \frac{1}{19}, \frac{3}{19}, \frac{4}{19}, \frac{7}{19}, \frac{8}{19}, \frac{9}{19}, \frac{10}{19}, \frac{11}{19}, \frac{13}{19}, \frac{14}{19}, \frac{17}{19}, \frac{18}{19}, 1 \right\}$, which satisfies \autoref{eq:translation_assume} ($\Leftrightarrow$  \autoref{eq:weak1_assume1} $\Leftrightarrow$ \autoref{eq:pBcond1weak}). The symbol $B_i^3$ is simplified to $B_i$ in this figure.}
  \label{fig:pBsp}
\end{figure}

Using Lemma~\autoref{theo:translation}, we can obtain the following lemma with regard to the $x$-component of control points:
\begin{lemma}
  If
  \begin{subequations}
    \begin{eqnarray}
      &&\Delta t_i=\Delta t_{i+n-p} \quad(i=0,\ldots,2p-1),\label{eq:weak1_assume1}\\
      &&x_i+L=x_{n-p+i}\quad(i=0,\ldots,p-1)\label{eq:weak1_assume2}
    \end{eqnarray}%
  \end{subequations}%
  are given, then a B-spline curve in \autoref{eq:Bcurve} satisfies
  \begin{eqnarray*}
    \frac{\diff^k x}{\diff t^k}(t_p) +\delta_{k0}L = \frac{\diff^k x}{\diff t^k}(t_n)\quad(k=0,\ldots,p-1),
  \end{eqnarray*}
  where $\delta_{k0}=1$ if $k=0$ and $0$ otherwise.
  \label{theo:weak1}
\end{lemma}
\def\mybecause#1{\quad(\text{$\because$ #1})}
\begin{proof}
\begin{eqnarray*}
  \frac{\diff^k x}{\diff t^k}(t_p)
  &=& \sum_{i=0}^{n-1}\frac{\diff^k B_i^p}{\diff t^k}(t_p)x_i \mybecause{Eq.~\autoref{eq:Bcurve}}\\
  &=& \sum_{i=0}^{p-1}\frac{\diff^k B_i^p}{\diff t^k}(t_p)x_i \mybecause{Non-zero functions at $t=t_p$ are $B_0^p,\ldots,B_{p-1}^p$}\\
  &=& \sum_{j=n-p}^{n-1}\frac{\diff^k B_{j-(n-p)}^p}{\diff t^k}(t_p)x_{j-(n-p)} \mybecause{$j:=n-p+i$}\\
  &=& \sum_{j=n-p}^{n-1}\frac{\diff^k B_{j-(n-p)+(n-p)}^p}{\diff t^k}(t_p+T)(x_{j-(n-p)+(n-p)}-L) \mybecause{Lemma~\ref{theo:translation} and Eq.~\autoref{eq:weak1_assume2}}\\
  &=& \sum_{j=0}^{n-1}\frac{\diff^k B_{j}^p}{\diff t^k}(t_n)x_{j}-L\sum_{j=0}^{n-1}\frac{\diff^k B_{j}^p}{\diff t}(t_n) \mybecause{Non-zero functions at $t=t_n$ are $B_{n-p}^p,\ldots,B_{n-1}^p$}\\
  &=& \frac{\diff^k x}{\diff t^k}(t_n)-L\delta_{k0} \mybecause{Eq.~\autoref{eq:Bcurve} and the partition of unity in Eq.~\autoref{eq:pou}}
\end{eqnarray*}
\end{proof}

Similarly, we have the following result regarding the $y$-component of control points:
\begin{lemma}
  If
  \begin{subequations}
    \begin{eqnarray*}
      &&\Delta t_i=\Delta t_{i+n-p} \quad(i=0,\ldots,2p-1),\label{eq:weak2_assume1}\\
      &&y_i=y_{n-p+i}\quad(i=0,\ldots,p-1)\label{eq:weak2_assume2}
    \end{eqnarray*}%
  \end{subequations}%
  are given, then a B-spline curve in \autoref{eq:Bcurve} satisfies
  \begin{eqnarray*}
    \frac{\diff^k y}{\diff t^k}(t_p) = \frac{\diff^k y}{\diff t^k}(t_n)\quad(k=0,\ldots,p-1).
  \end{eqnarray*}
  \label{theo:weak2}
\end{lemma}
\begin{proof}
  We may replace $x$ with $y$ and let $L$ be zero in the proof of Lemma~\ref{theo:weak1}.
\end{proof}

We can obtain the following theorem directly from Lemmas~\ref{theo:weak1} and \ref{theo:weak2}:
\begin{theorem}
  If the conditions
  \begin{subequations}
    \begin{eqnarray}
      &&\Delta t_i=\Delta t_{i+n-p} \quad(i=0,\ldots,2p-1),\label{eq:pBcond1weak}\\
      &&x_i+L=x_{n-p+i}\quad(i=0,\ldots,p-1),\label{eq:pBcond2weak}\\
      &&y_i=y_{n-p+i} \quad(i=0,\ldots,p-1)\label{eq:pBcond3weak}
    \end{eqnarray}%
  \end{subequations}%
  are given, then a B-spline curve in \autoref{eq:Bcurve} satisfies
  \begin{subequations}
    \begin{eqnarray}
      &&x(t_p) + L = x(t_n), \label{eq:pBprop1weak}\\
      &&\frac{\diff^k y}{\diff t^k}(t_p) = \frac{\diff^k y}{\diff t^k}(t_n)\quad(k=0,\ldots,p-1).\label{eq:pBprop2weak}
    \end{eqnarray}
  \end{subequations}%
  \label{theo:pBcurve_weak}%
\end{theorem}

By modifying Theorem~\ref{theo:pBcurve_weak}, we can obtain the following theorem:
\begin{theorem}
  If
  \begin{subequations}
    \begin{eqnarray}
      && t_i:=\frac{i}{n+p}\quad(i=0,\ldots,n+p),\label{eq:pBcond1}\\
      && x_i:=x_0+i H\quad(i=0,\ldots,n-1),\label{eq:pBcond2}\\
      && y_i=y_{n-p+i}\quad(i=0,\ldots,p-1)\label{eq:pBcond3}
    \end{eqnarray}%
    \label{eq:pBcond}%
  \end{subequations}
  are given, then a B-spline curve in \autoref{eq:Bcurve} satisfies
  \begin{subequations}
    \begin{eqnarray}
      &&x(t_p)=-\frac{L}{2},\quad x(t_n)=\frac{L}{2},\label{eq:pBprop1}\\
      &&\frac{\diff^k y}{\diff t^k}(t_p)=\frac{\diff^k y}{\diff t^k}(t_n)\quad(k=0,\ldots,p-1).\label{eq:pBprop2}
    \end{eqnarray}
  \end{subequations}
  %
  %
  \label{theo:pBcurve}
\end{theorem}
\begin{proof}
  Eq.~\autoref{eq:pBcond2} is a sufficient condition for \autoref{eq:pBcond2weak}, while \autoref{eq:pBcond3} is identical to \autoref{eq:pBcond3weak}. Therefore, \autoref{eq:pBprop2} follows from \autoref{eq:pBprop2weak} of Theorem~\ref{theo:pBcurve_weak}. So, we may show that the remaining equations in \autoref{eq:pBprop1} hold.
To this end, we first write down \autoref{eq:pBprop1} as follows:
  \begin{eqnarray*}
    &&x(t_p)=\sum_{i=0}^{n-1}B_i^p(t_p)x_i=\sum_{i=0}^{p-1}B_i^p(t_p)x_i=-\frac{L}{2},\\
    && x(t_n)=\sum_{i=0}^{n-1}B_i^p(t_n)x_i=\sum_{i=n-p}^{n-1}B_i^p(t_n)x_i=\frac{L}{2},
  \end{eqnarray*}
  where the supports of B-spline functions were considered. By substituting \eqref{eq:pBcond2} into the above equations, we have
  \begin{eqnarray}
    x_0+C_pH=-\frac{L}{2},\quad x_0+(C_p+n-p)H=\frac{L}{2},
    \label{eq:x0_hx}
  \end{eqnarray}
  where we used the partition of unity, i.e. $\sum_{i=0}^{p-1}B_i^p(t_p)=1$.
  Also, we defined a constant $C_p:=\sum_{i=0}^{p-1}B_i^p(t_p)i$, which depends on $p$ only, and used the fact that $\sum_{j=n-p}^{n-1} B_{j}^p(t_n)j = C_p + (n-p)$.\footnote{Since Lemma~\ref{theo:translation} is available, we can rewrite $C_p$ as
    \begin{eqnarray*}
      C_p=\sum_{i=0}^{p-1}B_{i+n-p}^p(t_p+t_n-t_p)i=\sum_{j=n-p}^{n-1} B_{j}^p(t_n)(j-(n-p))=\sum_{j=n-p}^{n-1} B_{j}^p(t_n)j-(n-p).
    \end{eqnarray*}
  } The linear equations in \autoref{eq:x0_hx} have the unique solutions if $n\ne p$, which is assumed in \autoref{eq:n>p}. Specifically, the solutions are given by
  \begin{eqnarray}
    x_0=-\frac{L}{2}-\frac{LC_p}{n-p}=-\frac{L}{2}-\frac{L(p-1)}{2(n-p)},\quad
    H=\frac{L}{n-p}.
    \label{eq:x0_hx_solution}
  \end{eqnarray}
  Here, we exploited $C_p=\frac{p-1}{2}$, which is proven as Lemma~\ref{theo:C_p} in \ref{s:proof_C_p}.
\end{proof}

\section{The value of $C_p$}\label{s:proof_C_p}

We have the following formula:
\begin{lemma}
  If the knots are uniform, a set of $n$ B-spline functions of degree $p$ (i.e. $B_0^p,\ldots,B_{n-1}^p$) satisfies
  \begin{eqnarray}
    C_p:=\sum_{i=0}^{p-1}B_i^p(t_p)i=\frac{p-1}{2}.
    \label{eq:C_p}
  \end{eqnarray}
  \label{theo:C_p}
\end{lemma}
\begin{proof}
By virtue of the uniform knots, the B-spline functions are cardinal and have the translation property
\begin{eqnarray}
  B^p_i(t_j)\equiv B^p_{i+k}(t_{j+k})
\end{eqnarray}
for any $i$, $j$, and $k$ as long as $B^p_{i+k}$ and $t_{j+k}$ can be defined. This property is similar to Lemma~\ref{theo:translation}. We will use it below.

We prove the statement by induction on $p$. First, $p=1$ holds because $C_1=B_0^1(t_1)\cdot 0=0$. Next, we assume \autoref{eq:C_p} holds. Then, we compute $C_{p+1}$ as 
\begin{eqnarray}
  C_{p+1}
  &=&\sum_{i=0}^{p}B_i^{p+1}(t_{p+1})i=\sum_{i=1}^{p}B_i^{p+1}(t_{p+1})i \mybecause{The first term is negligible}\nonumber\\
  &=&\sum_{i=1}^{p}\left( \frac{t_{p+1}-t_i}{t_{i+p+1}-t_i}B_i^p(t_{p+1})+\frac{t_{i+p+2}-t_{p+1}}{t_{i+p+2}-t_{i+1}}B_{i+1}^p(t_{p+1})\right)i \mybecause{Eq.~\autoref{eq:CoxdeBoor}}\nonumber\\
  &=&\sum_{i=1}^{p}\left( \frac{p+1-i}{p+1}B_i^p(t_{p+1})+\frac{i+1}{p+1}B_{i+1}^p(t_{p+1})\right)i \mybecause{Uniform knots, i.e. $t_i=\frac{i}{p+n}$}\nonumber\\
  &=&\underbrace{\sum_{i=1}^{p}B_i^p(t_{p+1})i}_{\displaystyle T_1}+\frac{1}{p+1}\underbrace{\left[\sum_{i=1}^p \left( B_{i+1}^p(t_{p+1})-B_{i}^p(t_{p+1})\right)i^2 \right]}_{\displaystyle T_2} +\frac{1}{p+1} \underbrace{\sum_{i=1}^p B_{i+1}^p(t_{p+1})i}_{\displaystyle T_3}.
\end{eqnarray}
The first term $T_1$ can be evaluated as
\begin{eqnarray*}
  T_1:=\sum_{i=1}^{p}B_{i-1}^p(t_p)i=\sum_{j=0}^{p-1}B_{j}^p(t_p)(j+1)=\sum_{j=0}^{p-1}B_{j}^p(t_p)j+\sum_{j=0}^{p-1}B_{j}^p(t_p).
\end{eqnarray*}
Here, the first term in the most RHS is $\frac{p-1}{2}$ by assumption, while the second term is 1 because of the partition of unity. Therefore, we have
\begin{eqnarray*}
  T_1=\frac{p+1}{2}.
\end{eqnarray*}

We can evaluate the second term $T_2$ as follows:
\begin{eqnarray*}
  T_2
  &:=&( B_{2}^p(t_{p+1})-B_{1}^p(t_{p+1}) )\cdot 1^2
  +( B_{3}^p(t_{p+1})-B_{2}^p(t_{p+1}) )\cdot 2^2
  +\cdots+( B_{p+1}^p(t_{p+1})-B_{p}^p(t_{p+1}) )\cdot p^2 \nonumber\\
  &=& B_{1}^p(t_{p+1})(0^2-1^2)
      +B_{2}^p(t_{p+1})(1^2-2^2)
      +\cdots+B_{p}^p(t_{p+1})\left((p-1)^2-p^2\right)
      +\underbrace{B_{p+1}^p(t_{p+1})}_{\text{Vanish}}p^2\nonumber\\
  &=&\sum_{i=1}^p B_{i}^p(t_{p+1})\left((i-1)^2-i^2\right)
  =\sum_{i=1}^p B_{i}^p(t_{p+1})-2\underbrace{\sum_{i=1}^p B_{i}^p(t_{p+1})i}_{\displaystyle T_4}\nonumber\\
  &=& 1-2T_4 \quad\text{($\because$ the partition of unity)},
\end{eqnarray*}
where
\begin{eqnarray*}
T_4 
  &:=& \sum_{i=1}^p B_{i-1}^p(t_{p})i
  = \sum_{j=0}^{p-1} B_{j}^p(t_{p})(j+1)
  = \sum_{j=0}^{p-1} B_{j}^p(t_{p})j+\sum_{j=0}^{p-1} B_{j}^p(t_{p}) \nonumber\\
  &=& \frac{p-1}{2}+1 \quad\text{($\because$ the assumption and the partition of unity)}.
\end{eqnarray*}

In regard to the third term $T_3$, we have
\begin{eqnarray*}
  T_3
  &:=&\sum_{i=1}^p B_{i}^p(t_{p})i=\sum_{i=1}^{p-1} B_{i}^p(t_{p})i + B_{p}^p(t_p)p\nonumber\\
  &=&\frac{p-1}{2} + 0 \quad\text{($\because$ the assumption and the support of $B_p^p$ is $[t_p,t_{2p+1}]$)}.
\end{eqnarray*}

Finally, we can obtain
\begin{eqnarray*}
  C_{p+1}=T_1+\frac{1}{p+1}(1-2T_4)+\frac{1}{p+1}T_3=\frac{p+1}{2}+\frac{1}{p+1}\left(1-2\frac{p+1}{2}\right)+\frac{1}{p+1}\frac{p-1}{2}=\frac{p}{2}.
\end{eqnarray*}

\end{proof}

\section{Proof for the quasi-periodicity}\label{s:proof_Mp}

When a surface current density $\vt$ is approximated with the vector basis function $\Mp$ as in \autoref{eq:approx_current}, i.e. $\vt(\vx)\approx\sum_h\sum_i\sum_j \vtcoef_{h,i,j} \Mp_{h,i,j}(\vx)$, we will prove that $\vt$ satisfies \autoref{eq:Pdivconf1} for any $\vx\in C_1$. (The proof for \autoref{eq:Pdivconf2} is similar.) To this end, we may show that each basis function $\Mp_{h,i,j}$ satisfies \autoref{eq:Pdivconf1} for $x_1=-L_1/2$ and any $x_2$ and $x_3$, that is,
\begin{eqnarray}
  \Mp_{h,i,j}(L_1/2, x_2, x_3)\cdot\bm{\tau}^- = \mathrm{e}^{\imath\beta_1}\Mp_{h,i,j}(-L_1/2, x_2, x_3)\cdot\bm{\tau}^-,
  \label{eq:Pdivconf1_Mp}
\end{eqnarray}
where $\bm{\tau}(-L_1/2,x_2,x_3)$ is simply denoted as $\bm{\tau}^-$.

We will see that \autoref{eq:Pdivconf1_Mp} is satisfied for all the three cases of $\Mp_{h,i,j}$ in the RHS of \autoref{eq:pMdiv} as follows:
\begin{enumerate}[label=(\roman*)]

\item $\Mp_{h,i,j}$ in the first case of \autoref{eq:pMdiv}, i.e. $h=1$, $0 \leq i < \qb_1$ and $0 \leq j < \mb_2$.

Since $x_1=L_1/2$ corresponds to $t_1=t_{1,n_1}=u_{1,\mb_1}$ from \autoref{eq:u=t} and the first property in Remark~\ref{theo:pBsurfprop}, we can write down $\Mp$ for $x_1=L_1/2$ as follows:
\begin{eqnarray*}
  \Mp_{1,i,j}(L_1/2,x_2,x_3)
  &=&\bm{V}_{1,i,j}(u_{1,\mb_1},t_2)+\mathrm{e}^{\imath\beta_1}\bm{V}_{1,i+\mb_1-\qb_1,j}(u_{1,\mb_1},t_2)\nonumber\\
  &=&\frac{1}{J(u_{1,\mb_1},t_2)}\left(B^{\qb_1}_i(u_{1,\mb_1})+\mathrm{e}^{\imath\beta_1} B^{\qb_1}_{i+\mb_1-\qb_1}(u_{1,\mb_1})\right)B^{\qb_2}_j(t_2)\frac{\partial\vx}{\partial t_1}(u_{1,\mb_1},t_2),
\end{eqnarray*}
where the definition of $\bm{V}_{1,i,j}$ in \autoref{eq:compatibleB1} was used. Since
\begin{eqnarray*}
  B^{\qb_1}_{i+\mb_1-\qb_1}(u_{1,\mb_1})=B^{\qb_1}_i\left(u_{1,\mb_1}-(u_{1,\mb_1}-u_{1,\qb_1})\right)
\end{eqnarray*}
follows from \autoref{eq:pbasis_knot1} and thus Lemma~\ref{theo:translation}, we have
\begin{eqnarray*}
  \Mp_{1,i,j}(L_1/2,x_2,x_3)
  =\frac{1}{J(u_{1,\mb_1},t_2)}\left(B^{\qb_1}_i(u_{1,\mb_1})+\mathrm{e}^{\imath\beta_1} B^{\qb_1}_i(u_{1,\qb_1})\right)B^{\qb_2}_j(t_2)\frac{\partial\vx}{\partial t_1}(u_{1,\mb_1},t_2).
\end{eqnarray*}
Here, the support of $B^{\qb_1}_i(u_{1,\mb_1})$, i.e. $[u_{1,i},u_{1,i+\qb_1+1}]$, for any $i\in[0,\qb_1)$ is included in $[u_{1,0},u_{1,2\qb_1}]$. Hence, $B^{\qb_1}_i(u_{1,\mb_1})$ vanishes by the assumption of $\mb_1\ge 2\qb_1$ in \autoref{eq:alotofknots1}. Therefore, we have
\begin{eqnarray*}
  \Mp_{1,i,j}(L_1/2,x_2,x_3) = \frac{\mathrm{e}^{\imath\beta_1}}{J(u_{1,\mb_1},t_2)} B^{\qb_1}_i (u_{1,\qb_1}) B^{\qb_2}_j(t_2) \frac{\partial\vx}{\partial t_1}(u_{1,\mb_1},t_2).
\end{eqnarray*}
Similarly, we can obtain
\begin{eqnarray*}
  \Mp_{1,i,j}(-L_1/2,x_2,x_3) = \frac{1}{J(u_{1,\qb_1},t_2)} B^{\qb_1}_i (u_{1,\qb_1}) B^{\qb_2}_j(t_2) \frac{\partial\vx}{\partial t_1}(u_{1,\qb_1},t_2).
\end{eqnarray*}
Here, $\frac{\partial\vx}{\partial t_1}(u_{1,\qb_1},t_2)=\frac{\partial\vx}{\partial t_1}(u_{1,\mb_1},t_2)$ and $J(u_{1,\qb_1},t_2)=J(u_{1,\mb_1},t_2)$ hold from both $u_{1,\qb_1}=t_{1,p_1}$ in \autoref{eq:u=t} and \autoref{eq:pBsurfprop2}. Consequently, we have
\begin{eqnarray*}
  \Mp_{1,i,j}(L_1/2,x_2,x_3) = \mathrm{e}^{\imath\beta_1} \Mp_{1,i,j}(-L_1/2,x_2,x_3). 
\end{eqnarray*}
Therefore, by applying $\bm{\tau}^-$ to the above equation, we can prove \autoref{eq:Pdivconf1_Mp}.

\item $\Mp_{h,i,j}$ in the second case of \autoref{eq:pMdiv}, i.e. $h=2$, $0 \leq i < \mb_3$ and $0 \leq j < \qb_4$.

Since the direction of $\Mp_{h,i,j}$ is determined by the factor $\frac{\partial\vx}{\partial t_h}$, the underlying $\Mp_{2,i,j}$ is perpendicular to the normal vector $\bm{\tau}^-$. Therefore, \autoref{eq:Pdivconf1_Mp} holds.

\item $\Mp_{h,i,j}$ in the third case of \autoref{eq:pMdiv}, i.e. $h=1$, $\qb_1\leq i < \mb_1-\qb_1$, $0 \leq j < \mb_2$ or $h=2$, $0 \leq i < \mb_3$ and $\qb_4 \leq j < \mb_4-\qb_4$.

In this case, any B-spline functions $B_i^{\qb_{2h-1}}$ and $B_j^{\qb_{2h}}$ in the definition of $\bm{V}_{h,i,j}$ in \autoref{eq:Nhij} vanish on $x_1=\pm L_1/2$. (\autoref{fig:pBsp} helps to understand this; the B-spline functions in black correspond to the underlying B-spline functions.) Therefore, \autoref{eq:Pdivconf1_Mp} holds.

\end{enumerate}
\qed

\section*{References}
\iftrue 

\else
\bibliographystyle{elsarticle/elsarticle-num}
\bibliography{ref}
\fi

\end{document}